\documentclass[a4paper,reqno,12pt]{amsart}
\usepackage{mathpazo}
\usepackage[euler-digits]{eulervm}
\usepackage[a4paper, scale={0.72,0.74}, marginratio={1:1}, footskip=7mm, headsep=10mm]{geometry}

\usepackage[american]{babel}

\usepackage{amsmath, amssymb, amsthm, graphicx, mathtools, bm, bbm, relsize, subcaption, enumitem}

\usepackage{hyperref}

\usepackage{perpage}

\usepackage{emptypage} 

\usepackage{caption, subcaption}

\usepackage[style=alphabetic, backend=bibtex8, citestyle=alphabetic, firstinits=true, maxnames=99, maxalphanames=5, sorting=nyt]{biblatex}
\DeclareFieldFormat[article, inbook, incollection, inproceedings, patent, thesis, unpublished]{title}{#1} 
\renewbibmacro{in:}{}

\theoremstyle{plain}
\newtheorem{thm}{Theorem}[section]
\newtheorem{lem}[thm]{Lemma}
\newtheorem{prop}[thm]{Proposition}
\newtheorem{coro}[thm]{Corollary}
\newtheorem{conj}[thm]{Conjecture}
\theoremstyle{definition}

\newtheorem{rem}[thm]{Remark}
\newtheorem{ex}[thm]{Example}

\makeatletter
\renewcommand{\@secnumfont}{\bfseries}
\renewcommand\section{\@startsection{section}{1}%
\z@{.7\linespacing\@plus\linespacing}{.5\linespacing}%
{\large\bfseries\scshape\centering}}
\renewcommand\subsection{\@startsection{subsection}{2}%
  \z@{.5\linespacing\@plus.7\linespacing}{-.5em}%
  {\bfseries\scshape}}
\renewcommand\subsubsection{\@startsection{subsubsection}{3}%
  \z@{.5\linespacing\@plus.7\linespacing}{-.5em}%
  {\scshape}}
\makeatother

\MakePerPage{footnote}
\makeatletter
\newcommand*{\myfnsymbolsingle}[1]{%
  \ensuremath{%
    \ifcase#1
    \or 
      \dagger
    \else 
      \@ctrerr  
    \fi
  }%
}   
\makeatother

\usepackage{alphalph}
\newalphalph{\myfnsymbolmult}[mult]{\myfnsymbolsingle}{}

\usepackage[dvipsnames]{xcolor}

\DeclarePairedDelimiter\abs{\lvert}{\rvert} 
\makeatletter
\let\oldabs\abs
\def\abs{\@ifstar{\oldabs}{\oldabs*}}

\newcommand{\firstmention}{\emph}
\newcommand{\N}{\mathbb{N}}
\newcommand{\Z}{\mathbb{Z}}
\newcommand{\R}{\mathbb{R}}
\newcommand{\C}{\mathbb{C}}
\newcommand{\sym}{S}
\newcommand{\freevecspace}{\mathbb{C}_{\mathbf{x}}^{n-1}S}
\newcommand{\syt}{\textnormal{SYT}}
\newcommand{\staircase}{\delta}
\newcommand{\young}{\mathcal{Y}}

\newcommand{\sort}{\textnormal{SN}}
\newcommand{\eg}{\textnormal{EG}}
\newcommand{\diag}{\textnormal{cor}}
\newcommand{\diagsorted}{\overline{\diag}}

\newcommand{\rev}{\textnormal{rev}}
\newcommand{\id}{\textnormal{id}}
\newcommand{\fin}{\textnormal{last}}
\newcommand{\finsorted}{\overline{\fin}}
\newcommand{\tab}{\textnormal{Tab}}
\newcommand{\interlacingtab}{\textnormal{IntTab}}

\newcommand{\Znonneg}{\mathbb{Z}_{\geq 0}}
\renewcommand{\i}{\mathrm{i}}
\newcommand{\Usorted}{\overline{U}}
\newcommand{\bmUsorted}{\overline{\bm{U}}}

\renewcommand{\P}{\mathbb{P}}
\newcommand{\diff}{\mathop{}\!\mathrm{d}}
\newcommand{\1}{\mathbbm{1}} 
\DeclareMathOperator{\e}{\mathrm{e}}
\newcommand{\RSK}{\mathsf{RSK}}
\newcommand{\Burge}{\mathsf{Bur}}

\newcommand{\Fro}{\mathcal{B}}

\renewcommand{\emptyset}{\varnothing}

\DeclareMathSymbol{\widehatsym}{\mathord}{largesymbols}{"62}

\renewcommand{\hat}{\widehat}

\begin{document}

\title[The oriented swap process and last passage percolation]{The oriented swap process \\ and last passage percolation}


\author[E.~Bisi]{Elia Bisi}
\address{Technische Universit\"at Wien \\
Institut f\"ur Stochastik und Wirtschaftsmathematik \\
E 105-07 \\
Wiedner Hauptstra{\ss}e 8-10 \\
1040 Wien, Austria}
\email{elia.bisi@tuwien.ac.at}

\author[F.~D.~Cunden]{Fabio Deelan Cunden}
\address{Dipartimento di Matematica\\
 Universit\`a degli Studi di Bari\\ 
 I-70125 Bari\\
  Italy}
\email{fabio.cunden@uniba.it}

\author[S.~Gibbons]{Shane Gibbons}
\address{School of Mathematics and Statistics\\
University College Dublin\\
Dublin 4, Ireland}
\email{shane.gibbons@ucdconnect.ie}

\author[D.~Romik]{Dan Romik}
\address{Department of Mathematics \\ University of California, Davis \\ One Shields Ave
\\ Davis, CA 95616 \\ USA}
\email{romik@math.ucdavis.edu}

\maketitle

\begin{abstract}
We present new probabilistic and combinatorial identities relating three random processes: the oriented swap process on $n$ particles, the corner growth process, and the last passage percolation model. We prove one of the probabilistic identities, relating a random vector of last passage percolation times to its dual, using the duality between the Robinson--Schensted--Knuth and Burge correspondences. A second probabilistic identity, relating those two vectors to a vector of `last swap times' in the oriented swap process, is conjectural. We give a computer-assisted proof of this identity for $n\le 6$ after first reformulating it as a purely combinatorial identity, and discuss its relation to the Edelman--Greene correspondence. The conjectural identity provides precise finite-$n$ and asymptotic predictions on the distribution of the absorbing time of the oriented swap process, thus conditionally solving an open problem posed by Angel, Holroyd and Romik.
\end{abstract}

\section{Introduction}
\label{sec:intro}

Randomly growing Young diagrams, and the related models known as \firstmention{Last Passage Percolation} (LPP) and the \firstmention{Totally Asymmetric Simple Exclusion Process} (TASEP), are intensively studied stochastic processes.
Their analysis has revealed many rich connections to the combinatorics of Young tableaux, longest increasing subsequences, the Robinson--Schensted--Knuth (RSK) algorithm, and related topics --- see for example~\cite[Chs.~4-5]{romik15}.

\firstmention{Random sorting networks} are another family of random processes.
Two main models, the \firstmention{Uniform Random Sorting Network} and the \firstmention{Oriented Swap Process} (OSP), have been analyzed~\cite{angelDauvergneEtAl17, angelHolroydRomikVirag07, angelHolroydRomik09, dauvergne18, dauvergneVirag18} and are known to have connections to the TASEP, last passage percolation, and also to staircase shape Young tableaux via the \firstmention{Edelman--Greene bijection}~\cite{edelmanGreene87}.

In this article we discuss a new and surprising meeting point between the aforementioned subjects.
In an attempt to address an open problem from~\cite{angelHolroydRomik09} concerning the absorbing time of the OSP, we discovered elegant distributional identities relating the oriented swap process to last passage percolation, and last passage percolation to itself.
We will prove one of the two main identities; the other one is a conjecture that we have been able to verify for small values of a parameter~$n$.
The analysis relies in a natural way on well-known notions of algebraic combinatorics, namely the RSK, Burge, and Edelman--Greene correspondences.

Our conjectured identity apparently requires new combinatorics to be explained, and has far-reaching consequences for the asymptotic behavior of the OSP as the number of particles grows to infinity, as will be explained in Subsection~\ref{subsec:absTimes}.

Most of the results in this paper were obtained in 2019 and announced in the proceedings of the 32nd Conference on Formal Power Series and Algebraic Combinatorics~\cite{bisiEtAl20}.
The present paper contains complete proofs, as well as additional material including:
\begin{itemize}
\item more detailed information about the RSK and Burge correspondences for random tableaux and their connection to distributional symmetries in last passage percolation;
\item some explicit formulas related to the conjectural identity and its connection to the largest eigenvalue of certain random matrices and Tracy-Widom distributions;
\item more details about the Edelman-Greene correspondence and its relation to the conjectural identity.
\end{itemize}

\subsection{Models}
\label{subsec:models}

The two main identities presented in this paper take the form 
\[
\bm{U}_n \overset{D}{=} \bm{V}_n \overset{D}{=}
\bm{W}_n \, ,
\]
where $ \overset{D}{=}$ denotes equality in distribution, and $\bm{U}_n$, $\bm{V}_n$, $\bm{W}_n$ are $(n-1)$-dimensional random vectors associated with the following three random processes.

\subsubsection*{The oriented swap process}

This process~\cite{angelHolroydRomik09} describes randomly sorting a list of~$n$~particles labelled $1,\ldots,n$.
At time $t=0$, particle labelled~$j$ is in position~$j$ on the finite integer lattice $[1,n]=\{1,\ldots,n\}$.
All pairs of adjacent positions $k,k+1$ of the lattice are assigned independent Poisson clocks.
The system then evolves according to the random dynamics whereby each pair of particles with labels $i,j$ occupying respective positions $k$, $k+1$ attempt to swap when the corresponding Poisson clock rings; the swap succeeds only if $i<j$, i.e.,\ if the swap increases the number of inversions in the sequence of particle labels.
The oriented swap process can also be interpreted as a continuous-time random walk on the Cayley graph of $S_n$ with adjacent swaps as generators (considered as a directed graph).
See Fig.~\ref{subfig:caleyGraph}.

We define the vector $\bm{U}_n = (U_n(1), \ldots, U_n(n-1))$ of \firstmention{last swap times} by
\begin{align*}
U_n(k) &:= \text{the last time $t$ at which a swap occurs between positions $k$ and $k+1$}.
\end{align*}
As explained in~\cite{angelHolroydRomik09}, the last swap times are related to the \firstmention{particle finishing times}: it is easy to see that $\max\{U_n(n-k), U_n(n-k+1)\}$ is the finishing time of particle $k$ (with the convention that $U_n(0)=U_n(n)=0$); see the equation on the last line of page 1988 of \cite{angelHolroydRomik09}.

\begin{figure}
\centering
\begin{subfigure}[b]{.5\linewidth}
\captionsetup{width=.9\textwidth}%
\centering
{\includegraphics[width=.85\columnwidth]{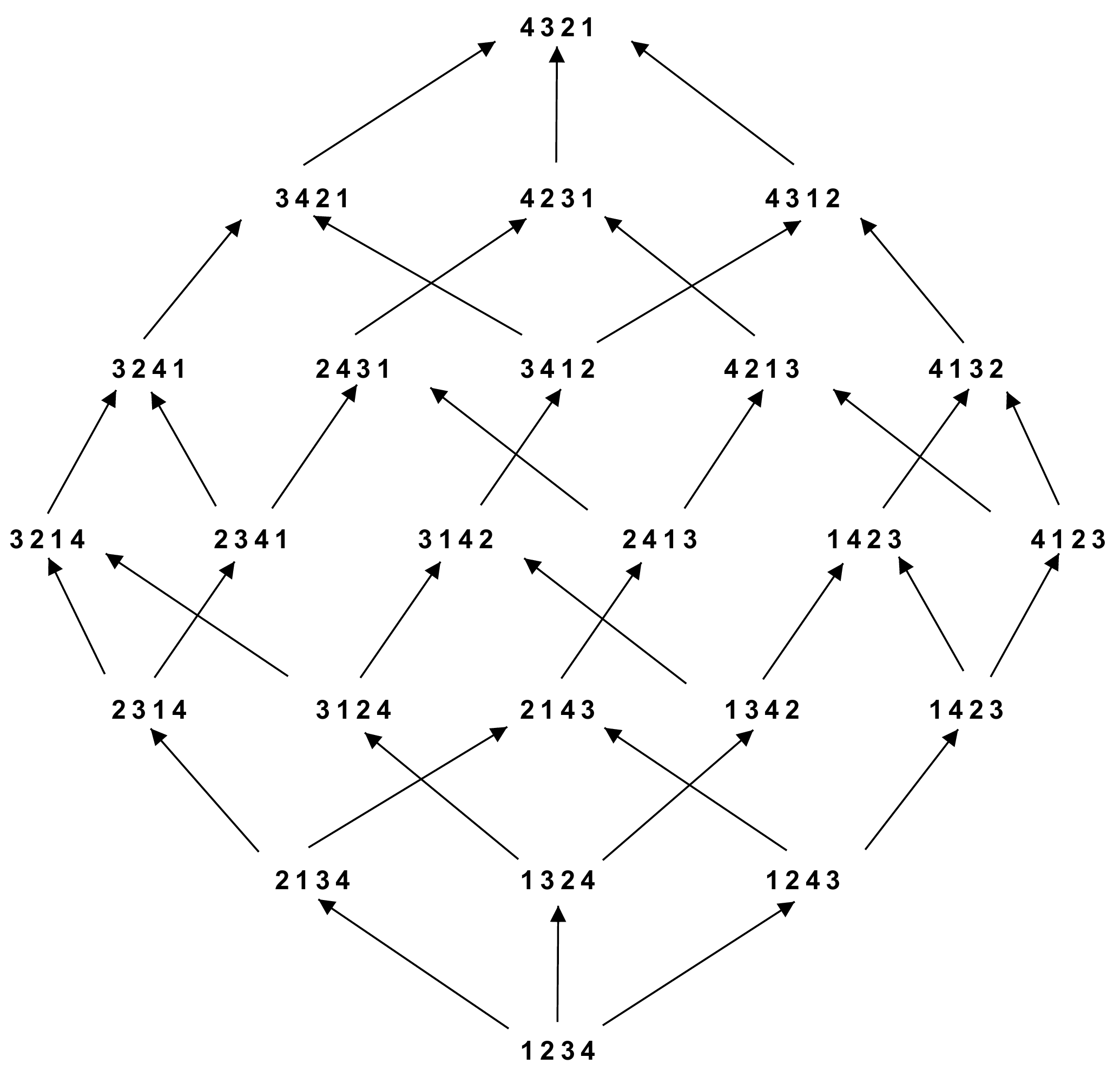}}
\subcaption{The Cayley graph of $S_n$ with Coxeter generators given by adjacent transpositions, for $n=4$.}
\label{subfig:caleyGraph}
\end{subfigure}%
\begin{subfigure}[b]{.5\linewidth}
\captionsetup{width=.9\textwidth}%
\centering
{\includegraphics[width=.8\columnwidth]{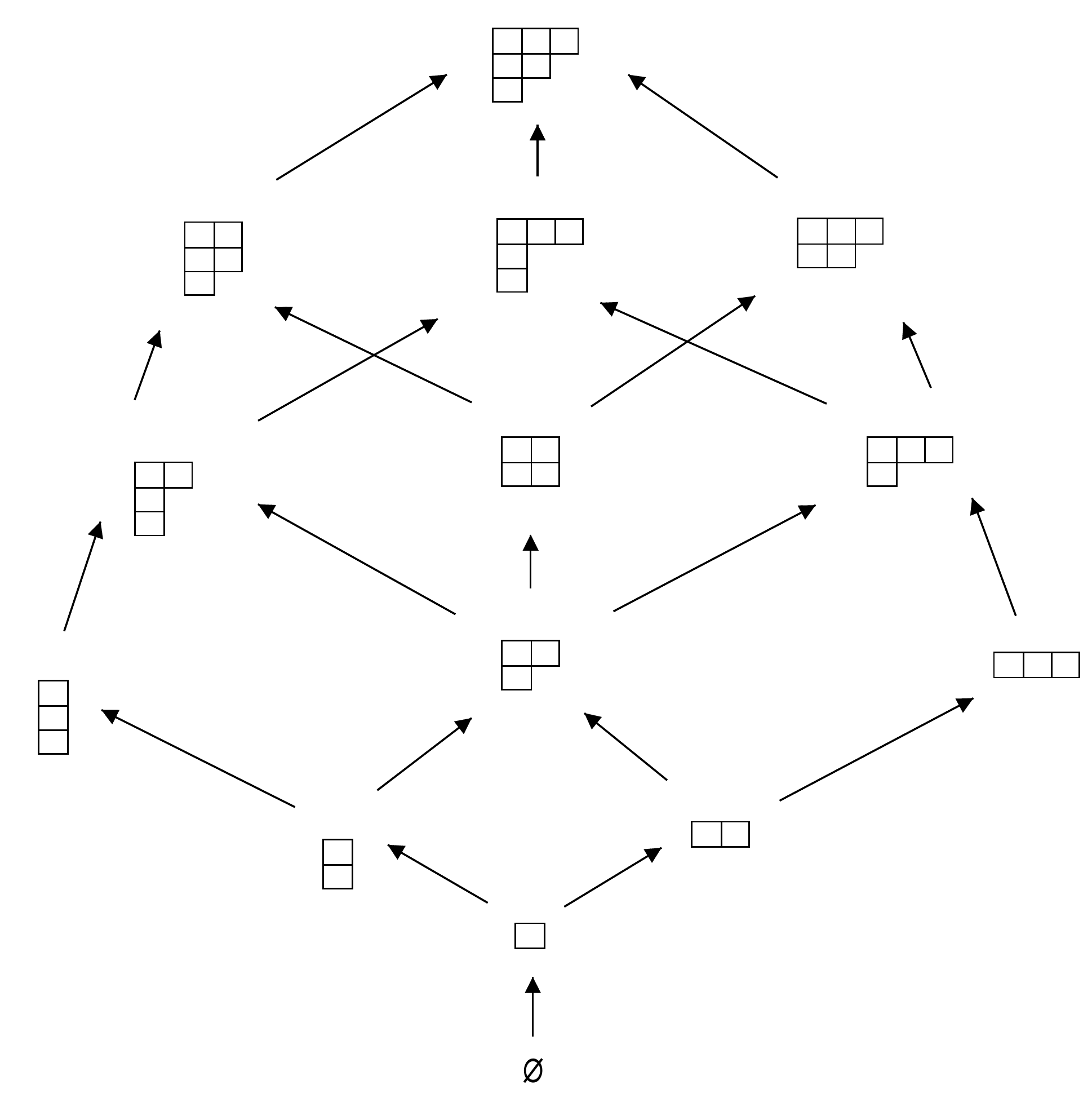}}
\subcaption{The Young sublattice $\young(\staircase_n)$ of all Young sub-diagrams of the staircase shape $\staircase_n$, for $n=4$.}
\label{subfig:youngGraph}
\end{subfigure}%
\caption{
Graphs related to the random walks representations of the oriented swap process and the randomly growing Young diagram model.
They can be regarded as \emph{directed} graphs, with edges directed from bottom to top.
}
\label{fig:graphs}
\end{figure}

\subsubsection*{Randomly growing a staircase shape Young diagram}

This process is a variant of the \firstmention{corner growth process}.
Starting from the empty Young diagram, boxes are successively added at random times, one box at each step, to form a larger diagram until the staircase shape $\staircase_n=(n-1,n-2,\ldots,1)$ is reached.
We identify each box of a Young diagram $\lambda$ with the position $(i,j)\in \N^2$, where $i$ and $j$ are the row and column index respectively.
All boxes are assigned independent Poisson clocks.
Each box $(i,j)\in \staircase_n$, according to its Poisson clock, attempts to add itself to the current diagram $\lambda$, succeeding if and only if $\lambda \cup \{(i,j)\}$ is still a Young diagram.
Notice that the randomly growing Young diagram model can be thought of as a continuous-time random walk, starting from $\emptyset$ and ending at $\staircase_n$, on the graph of Young diagrams contained in $\staircase_n$ (regarded in the obvious way as a directed graph).
See Fig.~\ref{subfig:youngGraph}.
Furthermore, note that every such random walk path is encoded by a standard Young tableau of shape $\staircase_n$, where the box added after $m$ steps is filled with $m$, for all $m=1,\dots, {n \choose 2}$.
For more details on this, see Subsection~\ref{subsec:staircaseTableaux} and, in particular, \eqref{eq:diagseq}.

We define $\bm{V}_n = (V_n(1), \ldots, V_n(n-1))$ as the vector that records when boxes along the $(n-1)$th anti-diagonal are added:
\[ 
V_n(k) := \text{the time at which the box at position $(n-k,k)$ is added.}
\]

\subsubsection*{The last passage percolation model}

This process describes the maximal time spent travelling from one vertex to another of the two-dimensional integer lattice along a directed path in a random environment. 
Let $(X_{i,j})_{i,j\ge 1}$ be an array of independent and identically distributed (i.i.d.) non-negative random variables, referred to as \firstmention{weights}.
For $(a,b), (c,d)\in \N^2$, define a \firstmention{directed lattice path} from $(a,b)$ to $(c,d)$ to be any sequence $\big( (i_k,j_k) \big)_{k=0}^m$ of minimal length $|c-a|+|d-b|$ such that $(i_0,j_0)=(a,b)$, $(i_m,j_m)=(c,d)$, and $\abs{i_{k+1}-i_k} + \abs{j_{k+1}-j_k} =1$ for all $0\leq k<m$.
We then define the \firstmention{Last Passage Percolation} (LPP) time from $(a,b)$ to $(c,d)$ as
\begin{equation}
\label{eq:lpp}
L(a,b;c,d) := \max_{\pi \colon (a,b)\to (c,d)} \sum_{(i,j)\in \pi} X_{i,j} \, ,
\end{equation}
where the maximum is over all directed lattice paths $\pi$ from $(a,b)$ to $(c,d)$.
It is immediate to see that LPP times starting at a fixed point, say $(1,1)$, satisfy the recursive relation
\begin{equation}
\label{eq:recLPP}
L(1,1;i,j)=\max\left\{L(1,1;i-1,j), L(1,1;i,j-1)\right\}+X_{i,j} \, ,
\qquad i,j\geq 1 \, ,
\end{equation}
with the boundary condition $L(1,1;i,j):=0$ if $i=0$ or $j= 0$.

If the weights $X_{i,j}$ are i.i.d.\ exponential random variables of rate $1$, the LPP model has a precise connection (see \cite[Ch.~4]{romik15}) with the corner growth process, whereby each random variable $L(1,1;i,j)$ is the time when box $(i,j)$ is added to the randomly growing Young diagram.
We can thus equivalently define $\bm{V}_n$ in terms of the last passage times between the fixed vertex $(1,1)$ and the vertices $(i,j)$ along the anti-diagonal line $i+j=n$:
\begin{equation}
\label{eq:def-vn}
\bm{V}_n = (L(1,1; n-1,1), L(1,1; n-2,2), \ldots, L(1,1; 1,n-1) ) \, .
\end{equation}
We refer to this as the \emph{point-to-line} LPP vector (see the illustration in Fig.~\ref{subfig:point-to-line} and the discussion in Subsection~\ref{subsec:absTimes} below).

Observe that $V_n(k)$ is the LPP time between two opposite vertices of the rectangular lattice $[1,n-k] \times [1,k]$, namely $(1,1)$ and $(n-k,k)$.
On the other hand, we can also consider the `dual' last passage times between the other two opposite vertices of the same rectangles, defining the vector $ \bm{W}_n = (W_n(1), \ldots, W_n(n-1))$ as
\begin{equation}
\label{eq:def-wn}
\bm{W}_n := (L(n-1,1; 1,1), L(n-2,1; 1,2), \ldots, L(1,1;1,n-1) ) \, .
\end{equation}
In this case, the starting and ending points for each last passage time vary simultaneously along the two lines $i=1$ and $j=1$, respectively.
We then refer to this vector $\bm{W}$ as the \emph{line-to-line} LPP vector (see Fig.~\ref{subfig:line-to-line}).

\begin{figure}
\centering
\begin{subfigure}[b]{.5\linewidth}
\centering
{\includegraphics[width=.7\columnwidth]{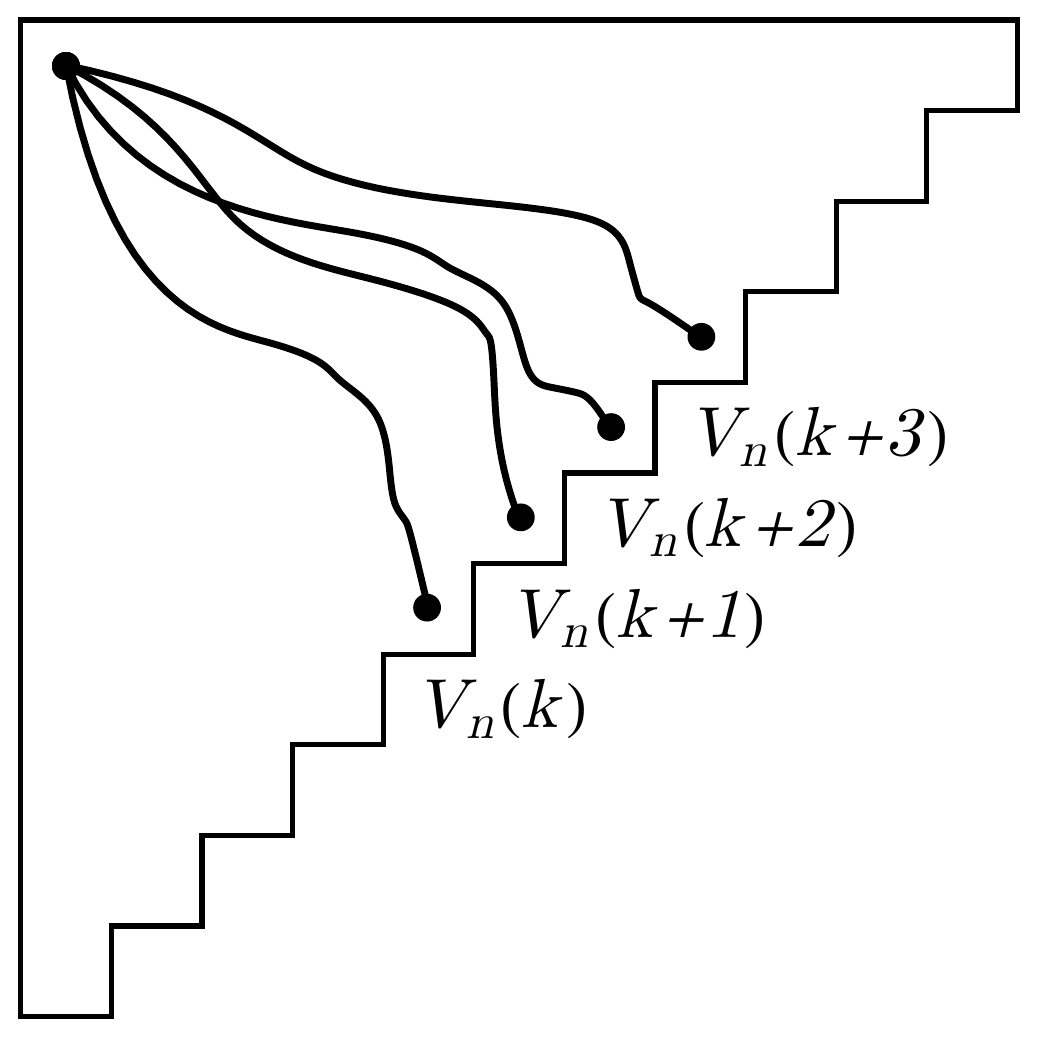}}
\subcaption{Point-to-line LPP vector $\bm{V}_n$.}
\label{subfig:point-to-line}
\end{subfigure}%
\begin{subfigure}[b]{.5\linewidth}
\centering
{\includegraphics[width=.895\columnwidth]{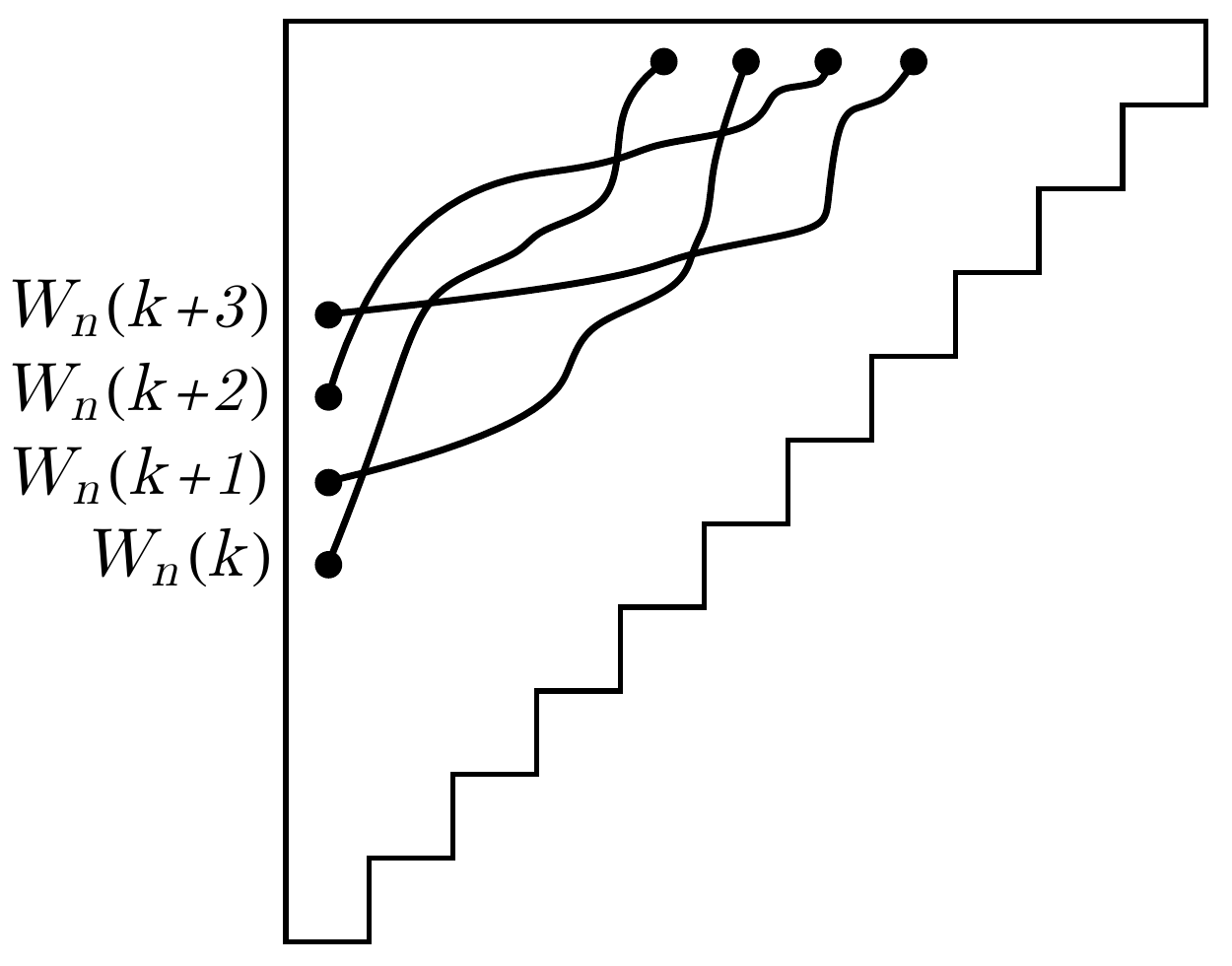}}
\subcaption{Line-to-line LPP vector $\bm{W}_n$.}
\label{subfig:line-to-line}
\end{subfigure}%
\caption{A schematic illustration of point-to-line and line-to-line last passage percolation vectors.
}
\label{fig:lpp-vectors}
\end{figure}

\subsection{Main results}
\label{subsec:results}

We can now state our results.

\begin{thm}
\label{thm:main-thm1}
$\bm{V}_n \overset{D}{=} \bm{W}_n$ for all $n\ge2$.
\end{thm}

\begin{conj}
\label{main-conj}
$\bm{U}_n \overset{D}{=} \bm{V}_n$  for all $n\ge2$.
\end{conj}

One might hope to prove Theorem~\ref{thm:main-thm1} and Conjecture~\ref{main-conj} by methods similar to those used to derive standard relations about last passage percolation. 
For example, the LPP recursive relation~\eqref{eq:recLPP} yields an explicit recursive formula for the joint density of $\bm{V}_n$,
\begin{equation}
\label{eq:jointDensityRecursive}
\begin{split}
p_{\bm{V}_n}(v_1,\dots,v_{n-1})
= &\int_0^{\min(v_1, v_2)} \diff y_1
\int_0^{\min(v_2, v_3)} \diff y_2
\cdots
\int_0^{\min(v_{n-2}, v_{n-1})} \diff y_{n-2} \\
&\small\times\exp\left\{\sum_{k=1}^{n-1} \big[\max(y_{k-1}, y_k) - v_k\big]\right\} p_{\bm{V}_{n-1}}(y_1,\dots, y_{n-2})
\end{split}
\end{equation}
for $n\geq 3$, with the convention that $y_0=y_{n-1}=0$, with the initial condition $p_{\bm{V}_{2}}(v) = \e^{-v} \1_{[0,\infty)}(v)$.
Surprisingly, formula~\eqref{eq:jointDensityRecursive} also holds for the line-to-line LPP vector $\bm{W}_n$ (as it must, by virtue of Theorem~\ref{thm:main-thm1}); Conjecture~\ref{main-conj} says that the joint density of $\bm{U}_n$ should also satisfy the same recursive relation.
However, we know of no simple recursive structure in the corresponding models to make possible such a direct proof.

Theorem~\ref{thm:main-thm1} and Conjecture~\ref{main-conj}  imply the equality of the one-dimensional marginal distributions
\begin{equation}
\label{eq:one-dim-marginal}
U_n(k) \overset{D}{=} V_n(k) \overset{D}{=} W_n(k),\quad\text{for all $1\le k\le n-1$, $n\geq2$}.
\end{equation}
The identity  $U_n(k) \overset{D}{=} V_n(k)$ was proved by Angel, Holroyd and Romik~\cite{angelHolroydRomik09} using a connection between the oriented swap process, the TASEP and the corner growth model.
The identity $V_n(k) \overset{D}{=} W_n(k)$ follows immediately from the observation that these two variables are the LPP times, on the same i.i.d.\ environment $(X_{i,j})_{i,j\ge 1}$, between two pairs of opposite vertices of the same rectangular lattice $[1,n-k] \times [1,k]$.

It is also easy to see that the following two-dimensional marginals coincide
\begin{equation}
\label{eq:special-marginal}
\left(U_n(1), U_n(n-1)\right) 
\overset{D}{=}
\left(V_n(1), V_n(n-1)\right)
\overset{D}{=}
\left(W_n(1), W_n(n-1)\right),
\end{equation}
for all $n \ge 2$.
The second equality actually holds almost surely, since  $\bm{V}_n$ and $\bm{W}_n$ are LPP vectors on the same environment $(X_{i,j})_{i,j\ge 1}$.
To check the first identity, observe that $U_n(n-1)$ and $U_n(1)$ are the finishing times of the first and last particle in the OSP, respectively.
Particle labelled $1$ (resp. $n$) jumps $n-1$ times only to the right (resp.\ to the left), always with rate $1$.
All these jumps are independent of each other, except the one that occurs when particles $1$ and $n$ are adjacent and swap.
Hence, $\left(U_n(1),U_n(n-1)\right)$ is jointly distributed as $(\Gamma+X, \Gamma'+X)$ where $\Gamma,\Gamma'$ are independent with $\operatorname{Gamma}(n-2,1)$ distribution and $X$ has $\operatorname{Exp}(1)$ distribution and is independent of $\Gamma,\Gamma'$.
This is the same joint distribution of the LPP times $\left(V_n(1), V_n(n-1)\right)$.

Theorem~\ref{thm:main-thm1} is proved in Section~\ref{sec:LPP}.
As we will see, the distributional identity $\bm{V}_n \overset{D}{=} \bm{W}_n$ arises as a special case of a more general family of identities (Theorem~\ref{thm:UpDownLPP}) involving LPP times between pairs of opposite vertices in rectangles $[1,i]\times [1,j]$, where each $(i,j)$ belongs to the so-called border strip of a Young diagram.
This result is, in turn, a consequence of the duality between the RSK and Burge correspondences, and holds also in the discrete setting where the weights $X_{i,j}$ follow a geometric distribution.
Theorem~\ref{thm:main-thm1} can be seen as a special case of a ``shift-invariance'' symmetry, conjectured in~\cite{borodinGorinWheeler19} for a variety of integrable stochastic systems, and recently proved in full generality in~\cite[Theorem~1.2]{dauvergne21}.

On the other hand, the conjectural equality in distribution between $\bm{U}_n$ and $\bm{V}_n$ remains mysterious, but we made some progress towards understanding its meaning by reformulating it as an algebraic-combinatorial identity that is of independent interest.
\begin{conj} 
\label{main-conj-reformulated}
For $n\ge 2$ we have the identity of vector-valued generating functions
\begin{equation}
\label{eq:comb-iden}
\sum_{t\in \syt(\staircase_n)} f_t(x_1,\ldots,x_{n-1}) \sigma_t
= \sum_{s\in \sort_n} 
g_s(x_1,\ldots,x_{n-1}) \pi_s \, .
\end{equation}
\end{conj}
Precise definitions and examples will be given in Section~\ref{sec:comb-iden}, where we will prove the equivalence between Conjectures~\ref{main-conj} and~\ref{main-conj-reformulated}.
For the moment, we only remark that the sums on the left-hand and right-hand sides of~\eqref{eq:comb-iden} range over the sets of staircase shape standard Young tableaux $t$ and sorting networks $s$ of order~$n$, respectively; $f_t$ and $g_s$ are certain rational functions, and $\sigma_t$, $\pi_s$ are permutations in the symmetric group $\sym_{n-1}$ that are associated with $t$ and $s$. 

The identity \eqref{eq:comb-iden} reduces the proof of $\bm{U}_n \overset{D}{=} \bm{V}_n$ for fixed~$n$ to a concrete finite computation.
This enabled us to provide a computer-assisted verification of Conjecture~\ref{main-conj} for $4\le n\le 6$ (the cases $n=2,3$ can be checked by hand) and thus prove the following:
\begin{thm}
\label{thm:main-thm2}
$\bm{U}_n \overset{D}{=} \bm{V}_n$  for  $2\le n\le 6$.
\end{thm}

\subsection{Absorbing times and random matrices}
\label{subsec:absTimes}

Conjecture~\ref{main-conj} has an important consequence in the asymptotic analysis of the oriented swap process.
Specifically, it addresses the open problem posed in~\cite{angelHolroydRomik09} (see also~\cite[Ex.~5.22(e), p.~331]{romik15}) about the limiting distribution, as $n\to\infty$, of
\begin{equation}
\label{eq:absTimeOSP}
U^{\max}_n := \max_{1\le k\le n-1} U_n(k) \, ,
\end{equation}
i.e.\ the absorbing time of the OSP on $n$ particles.

Observe first that the random variable
\begin{equation}
\label{eq:pointToLineLPP}
V^{\max}_n := \max_{1\leq k\leq n-1} V_k = \max_{\substack{\pi \colon (1,1)\to (a,b), \\ a+b=n}} \sum_{(i,j)\in \pi} X_{i,j} \, ,
\end{equation}
where $(X_{i,j})_{i,j\ge 1}$ are i.i.d.\ exponential random variables of rate $1$, represents the time until the staircase shape $\delta_n$ is reached in the corner growth process.
As the last expression in~\eqref{eq:pointToLineLPP} points out, it can also be seen as the maximal time spent travelling from the point $(1,1)$ to any point of the line $\{(a,b)\in \N^2\colon a+b = n\}$ along a directed path in an exponentially distributed random environment.
This variable has been referred to as the \emph{point-to-line} last passage percolation time and has been an object of study in the literature.

It is known that the point-to-line LPP time $V^{\max}_n$ with exponential weights is exactly distributed as the largest eigenvalue $\lambda^{(n)}_{\max}$ of an $n\times n$ random matrix drawn from the Laguerre Orthogonal Ensemble (LOE) --- see e.g.~\cite{baikRains01a, fitzgeraldWarren20}.
In the limit as $n\to\infty$, $\lambda^{(n)}_{\max}$ features KPZ fluctuations of order $n^{1/3}$ and has the $\beta=1$ Tracy--Widom distribution (first obtained by Tracy and Widom in~\cite{tracyWidom96}) as its limiting law; see~\cite[Theorem~1.1]{johnstone01}.

The asymptotic distribution of the point-to-line LPP time and some closely related random variables have also been studied independently of its connection with random matrix theory.
Baik and Rains~\cite{baikRains01b} proved a limit theorem for a conceptually related model, 
i.e.\ the length of the longest increasing subsequence of random involutions.
Borodin, Ferrari, Pr\"ahofer and Sasamoto~\cite{sasamoto05, borodinFerrariPrahoferSasamoto07} studied the asymptotic distribution of the TASEP with particle-hole alternating (``flat'') initial configuration; using the usual correspondence between LPP and TASEP, this can be viewed as an analogous result for the point-to-line last passage percolation model.
More recently, Bisi and Zygouras~\cite[Theorem~1.1]{bisiZygouras19b} obtained the asymptotics of the point-to-line LPP time~\eqref{eq:pointToLineLPP} using the determinantal structure provided by symplectic Schur functions.

On the other hand, modulo Conjecture~\ref{main-conj}, we have that 
\begin{equation}
\label{eq:weakconj-eq-dist}
U^{\max}_n\overset{D}{=} V^{\max}_n.
\end{equation}
The precise knowledge of the (finite $n$ and asymptotic) distribution of $V^{\max}_n$ thus extends to $U^{\max}_n$.
\begin{coro}
\label{thm:absTimeOSP}
Let $U^{\max}_n$ be the absorbing time of the OSP on $n$ particles, as in~\eqref{eq:absTimeOSP}. Then, assuming Conjecture~\ref{main-conj}:
\begin{enumerate}[label=(\roman*)]
\item for any $n\geq 2$, $t\geq0$,
\begin{equation}
\label{eq:LOE}
\P\left(U^{\max}_n\leq t\right)=\frac{1}{C_{n}}\int_{[0,t]^{n-1}}\prod_{1\leq i<j\leq n-1}\abs{y_i-y_j} \prod_{i=1}^{n-1} \e^{-y_i} \diff y_i \, ,
\end{equation}
where $C_{n}$ is a normalization constant; 
\item the following limit in distribution holds:
\begin{equation}
\label{eq:F1}
\frac{U^{\max}_n - 2n}{(2n)^{1/3}} \xrightarrow{n\to\infty} F_1 \, ,
\end{equation}
where $F_1$ is the $\beta=1$ Tracy--Widom law.
\end{enumerate}
\end{coro}

The integral formula in~\eqref{eq:LOE} is the distribution function of the largest eigenvalue in the Laguerre Orthogonal Ensemble (LOE).
It occurs in the following way.
Let $Y$ be an $n\times (n-1)$  matrix with entries that are independent real Gaussian random variables with mean zero and variance $1/2$.
Then the right-hand side in \eqref{eq:LOE} is the probability that the largest eigenvalue of $YY^T$  (also called a real Wishart matrix) is less than $t$ --- see e.g.~\cite[\S~3.2]{forrester10}.

As mentioned in the extended abstract version of this paper~\cite{bisiEtAl20}, the distributional limit~\eqref{eq:F1} answers the open problem posed in~\cite{angelHolroydRomik09} about the asymptotic distribution of the absorbing time of the OSP, conditionally on Conjecture~\ref{main-conj}.
Following the appearance of the extended abstract version of this paper, Bufetov, Gorin and Romik found a way to derive~\eqref{eq:weakconj-eq-dist} (and therefore deduce \eqref{eq:LOE} and \eqref{eq:F1}) by proving a weaker version of our Conjecture~\ref{main-conj} that equates the joint distribution functions of the random vectors $\bm{U}_n$ and $\bm{V}_n$  for `diagonal points', i.e.\ points  $(t,t,\ldots,t)\in \R^{n-1}$.
This is of course sufficient to imply equality in distribution of the maxima of the coordinates of the respective vectors.
Thus, the open problem from~\cite{angelHolroydRomik09} is now settled.

\begin{thm}[Bufetov-Gorin-Romik (2020)~\cite{bufetovGorinRomik21}]
The relations \eqref{eq:weakconj-eq-dist}, \eqref{eq:LOE} and \eqref{eq:F1} are true unconditionally.
\end{thm}

\section{Equidistribution of LPP times and dual LPP times \\ along border strips}
\label{sec:LPP}

The goal of this section is to prove Theorem~\ref{thm:main-thm1}.
We will in fact prove a more general statement (Theorem~\ref{thm:UpDownLPP}), which establishes the joint distributional equality between LPP times and dual LPP times along the so-called `border strips'.

\subsection{LPP and dual LPP tableaux}
\label{subsec:LPP-dualLPP}

We first fix some terminology.
We say that $(i,j)$ is a \firstmention{border box} of a Young diagram $\lambda$ if $(i+1,j+1)\notin \lambda$, or equivalently if $(i,j)$ is the last box of its diagonal.
We refer to the set of border boxes of $\lambda$ as the \emph{border strip} of $\lambda$.
We say that $(i,j)\in \lambda$ is a \emph{corner} of $\lambda$ if $\lambda\setminus\{(i,j)\}$ is a Young diagram.
Note that every corner is a border box.
We refer to any array $x=\{x_{i,j}\colon (i,j)\in\lambda\}$ of non-negative real numbers as a \emph{tableau} of shape $\lambda$.
We call such an $x$ an \emph{interlacing tableau} if its diagonals interlace, in the sense that
\begin{equation}
\label{eq:interlacing}
x_{i-1,j} \leq x_{i,j} \quad \text{if } i>1
\qquad\quad \text{and} \qquad\quad
x_{i,j-1} \leq x_{i,j} \quad \text{if } j>1
\end{equation}
for all $(i,j)\in\lambda$, or equivalently if its entries are weakly increasing along rows and columns.
As a reference, see the tableaux in Fig.~\ref{fig:RSK-Burge}.
Their common shape $\lambda=(4,3,3,3,1)$ has border strip $\Fro=\{(1,4), (1,3), (2,3), (3,3), (4,3), (4,2), (4,1), (5,1)\}$, and corners $(1,4), (4,3), (5,1)$; the two tableaux on the right are interlacing.

Throughout this section, $\lambda$ will denote an arbitrary but fixed Young diagram.
Let now $X$ be a \emph{random} tableau of shape $\lambda$ with i.i.d.\ non-negative random entries $X_{i,j}$.
We can then define the associated LPP time $L(a,b;c,d)$ on $X$ between two boxes $(a,b),(c,d)\in\lambda$ as in~\eqref{eq:lpp}.
We will mainly be interested in the special $\lambda$-shaped tableaux
$L = (L_{i,j})_{(i,j) \in \lambda}$ and $L^* = (L^*_{i,j})_{(i,j) \in \lambda}$, which we respectively call the \firstmention{LPP tableau} and the \firstmention{dual LPP tableau}, defined by
\begin{equation}
\label{eq:LPP&dualLPP}
L_{i,j} := L(1,1; i,j)
\qquad \text{and} \qquad
L^*_{i,j} := L(i,1; 1,j) \, ,
\qquad
\text{for } (i,j)\in \lambda \, .
\end{equation}
It is easy to see from the definitions that $L$ and $L^*$ are both (random) interlacing tableaux.

Now, it is evident that, for each $(i,j)\in \lambda$, the distributions of $L_{i,j}$ and $L^*_{i,j}$ coincide.
However, the joint distributions of $L$ and $L^*$ do not coincide in general.
\begin{prop}
\label{prop:LPP&dualLPP_joint}
Let $X$ be a Young tableau of shape $\lambda$ with i.i.d.\ non-deterministic\footnote{In the sense that their common distribution is not a Dirac measure.} entries.
Then the corresponding LPP and dual LPP tableaux $L$ and $L^*$ follow the same law if and only if $\lambda$ is a hook shape (a Young diagram with at most one row of length $> 1$).
\end{prop}
\begin{proof}
If $\lambda$ is a hook shape, then $L=L^*$ almost surely; in particular, the two tableaux have the same law.
Suppose now that $\lambda$ is not a hook shape, i.e.\ $(2,2)\in \lambda$.
By definition of $L$ and $L^*$, we have
\begin{gather*}
L_{1,1} = L^*_{1,1} = X_{1,1} \, , \qquad\quad
L_{1,2} = L^*_{1,2} = X_{1,1} + X_{1,2} \, , \qquad\quad
L_{2,1} = L^*_{2,1} = X_{1,1} + X_{2,1} \, , \\
L_{2,2} = X_{1,1} + \max(X_{1,2},X_{2,1}) + X_{2,2} \, , \qquad\quad
L^*_{2,2} = X_{2,1} + \max(X_{1,1},X_{2,2}) + X_{1,2} \, .
\end{gather*}
It immediately follows that
\begin{align*}
L_{2,2} - L_{1,2} - L_{2,1} + L_{1,1}
&= X_{2,2} - \min(X_{1,2},X_{2,1}) \, , \\
L^*_{2,2} - L^*_{1,2} - L^*_{2,1} + L^*_{1,1}
&= \max(0,X_{2,2}-X_{1,1}) \, .
\end{align*}
As by hypothesis the $X_{i,j}$'s are non-deterministic, there exists $t\in\R$ such that their (common) cumulative distribution function $F$ satisfies $0<F(t)<1$.
We then have, by independence of the $X_{i,j}$'s, that
\[
\begin{split}
\P(L_{2,2} - L_{1,2} - & L_{2,1} + L_{1,1} <0)
= \P(X_{2,2} < \min(X_{1,2},X_{2,1})) \\
\geq \, &\P(X_{2,2} \leq t, \, X_{1,2} > t, \, X_{2,1} > t)
= F(t) (1-F(t))^2 >0 \, .
\end{split}
\]
On the other hand,
\[
\P(L^*_{2,2} - L^*_{1,2} - L^*_{2,1} + L^*_{1,1} <0)
= \P(\max(0,X_{2,2}-X_{1,1})<0) =0 \, .
\]
It follows that $L_{2,2} - L_{1,2} - L_{2,1} + L_{1,1}$ and $L^*_{2,2} - L^*_{1,2} - L^*_{2,1} + L^*_{1,1}$ are not equally distributed.
In particular, $L$ and $L^*$ do not follow the same joint law.
\end{proof}

The main result of this section is that certain distributional identities between LPP and dual LPP do hold as long as the common distribution of the weights is geometric or exponential:

\begin{thm}
\label{thm:UpDownLPP}
Let $X$ be a Young tableau of shape $\lambda$ with i.i.d.\ geometric or i.i.d.\ exponential weights.
Then the border strip entries (and in particular the corner entries) of the corresponding LPP and dual LPP tableaux $L$ and $L^*$ have the same joint distribution.
\end{thm}

Theorem~\ref{thm:main-thm1} immediately follows from Theorem~\ref{thm:UpDownLPP} applied to tableaux of staircase shape $(n-1,n-2,\dots,1)$, since in this case the coordinates of $\bm{V}_n$ and $\bm{W}_n$ are precisely the corner entries of $L$ and $L^*$, respectively.

\begin{rem}
In a similar vein to how Proposition~\ref{prop:LPP&dualLPP_joint} illustrates the limits of what types of identities in distribution might be expected to hold, 
note as well that, in general, Theorem~\ref{thm:UpDownLPP} fails to hold if the weights are not geometric nor exponential.
For example, consider the square shape $\lambda=(2,2)$ and assume the $X_{i,j}$'s are uniformly distributed on $\{0,1\}$.
Then, we have that
\[
\P(L_{1,2}=2, \, L_{2,2}=3, \, L_{2,1}=1)
= \P(X_{1,1}=X_{1,2}=X_{2,2}=1, \, X_{2,1}=0)
= 2^{-4} \, ,
\]
but
\[
\P(L^*_{1,2}=2, \, L^*_{2,2}=3, \, L^*_{2,1}=1)
=0 \, .
\]
Thus $L$ and $L^*$, even when restricted to the border strip $\Fro = \{(2,1),(2,2),(1,2)\}$ of $\lambda$, are not equally distributed.
\end{rem}

\subsection{RSK and Burge correspondences}
\label{subsec:RSK&Burge}

We will prove Theorem~\ref{thm:UpDownLPP} via an extended version of two celebrated combinatorial maps, the Robinson--Schensted--Knuth and Burge correspondences, acting on arrays of arbitrary shape $\lambda$.

We denote by $\tab_{\Z_{\geq 0}}(\lambda)$ the set of tableaux of shape $\lambda$ with non-negative integer entries, and by $\interlacingtab_{\Z_{\geq 0}}(\lambda)$ the subset of interlacing tableaux, in the sense of~\eqref{eq:interlacing}.
Let $\Pi^{(k)}_{m,n}$ be the set of all unions of $k$ disjoint non-intersecting directed lattice paths $\pi_1, \dots, \pi_k$ with $\pi_i$ starting at $(1,i)$ and ending at $(m,n-k+i)$.
Similarly, let $\Pi^{*(k)}_{m,n}$ be the set of all unions of $k$ disjoint non-intersecting directed lattice paths $\pi_1,\dots,\pi_k$ with $\pi_i$ starting at $(m,i)$ and ending at $(1,n-k+i)$.
\begin{thm}[\cite{bisiOConnellZygouras20, greene74, krattenthaler06}]
\label{thm:RSK}
Let $\lambda$ be a Young diagram with border strip $\Fro$.
There exist two bijections 
\begin{align*}
\RSK \colon \tab_{\Z_{\geq 0}}(\lambda)&\to \interlacingtab_{\Z_{\geq 0}}(\lambda)\, ,
& x = \{x_{i,j} \colon (i,j)\in \lambda\}
&\xmapsto{\RSK}
r = \{r_{i,j} \colon (i,j)\in \lambda\} \, , \\
\Burge \colon\tab_{\Z_{\geq 0}}(\lambda)&\to \interlacingtab_{\Z_{\geq 0}}(\lambda)\, ,
& x = \{x_{i,j} \colon (i,j)\in \lambda\}
&\xmapsto{\,\Burge\,}
b = \{b_{i,j} \colon (i,j)\in \lambda\} \, ,
\end{align*}
called the Robinson--Schensted--Knuth and Burge correspondences, that are characterized (in fact defined) by the following relations:
for any $(m,n)\in \Fro$ and $1\leq k\leq \min(m,n)$,
\begin{align}
\label{eq:RSK_Greene}
\sum_{i=1}^k r_{m-i+1,n-i+1}
&= \max_{\pi \in\Pi^{(k)}_{m,n}} \sum_{(i,j)\in \pi} x_{i,j} \, , \\ \label{eq:Burge_Greene}
\sum_{i=1}^k b_{m-i+1,n-i+1}
&= \max_{\pi \in\Pi^{*(k)}_{m,n}} \sum_{(i,j)\in \pi} x_{i,j} \, .
\end{align}
\end{thm}

The RSK correspondence was introduced by Robinson, Schensted, and Knuth --- see the classic paper~\cite{knuth70} as well as the modern presentation in~\cite[\S~7.11]{stanley99}.
The Burge correspondence is one of the bijections presented in~\cite{burge74} --- see also~\cite[App.~A]{fulton97}.
In the usual setting, both these maps are regarded as bijections between non-negative integer matrices $x$ and a pair $(P,Q)$ of semistandard Young tableaux of the same shape.
They are defined, respectively, in terms of \emph{row insertion} and \emph{column insertion}, two combinatorial algorithms that `insert' a given positive integer into a given semistandard Young tableau, yielding a new semistandard Young tableau with one extra box --- see~\cite[\S~1.1 and~A.2]{fulton97}.

Theorem~\ref{thm:RSK} presents the RSK and Burge correspondences, in a somewhat untraditional way, as bijections between tableaux and interlacing tableaux with non-negative integer entries.
This generalization goes through an alternative construction of these maps in terms of $(\max, \min,+,-)$ operations on the elements of the input tableau, as described in~\cite[\S~2]{bisiOConnellZygouras20} (therein, the bijections are further extended to tableaux with real entries).
Relations~\eqref{eq:RSK_Greene}-\eqref{eq:Burge_Greene} can be then regarded as an extension of so-called Greene's theorem~\cite{greene74}.
The paper of Krattenthaler~\cite{krattenthaler06} contains all the details of the constructions leading to Theorem~\ref{thm:RSK}, even though expressed in a slightly different language.
For the reader's convenience we translate the results of~\cite{krattenthaler06} into our setting in Appendix~\ref{app:RSK_Burge}.

For the proof of Theorem~\ref{thm:UpDownLPP}, we will be  using the extremal cases $k=1$ and $k=\min(m,n)$ of~\eqref{eq:RSK_Greene} and~\eqref{eq:Burge_Greene}.

The case $k=1$ explains the connection between the outputs of the RSK (respectively, Burge) correspondence and the LPP (respectively, dual LPP) times.
More precisely, we have that
\begin{equation}
\label{eq:LPP_RSK_Burge}
r_{m,n}
= \max_{\pi\colon (1,1) \to (m,n)} \sum_{(i,j)\in \pi} x_{i,j}
\qquad \text{and} \qquad
b_{m,n}
= \max_{\pi\colon (m,1) \to (1,n)} \sum_{(i,j)\in \pi} x_{i,j} \, ,
\end{equation}
for all $(m,n)$ on the border strip $\Fro$ on $\lambda$.

On the other hand, taking $k=\min(m,n)$ in Theorem~\ref{thm:RSK}, it is easy to see that the maxima in~\eqref{eq:RSK_Greene} and~\eqref{eq:Burge_Greene} become both equal to the \emph{same} `rectangular sum' ${\rm Rec}_{m,n}(x)$ of inputs:
\begin{equation}
\label{eq:RSK_Burge_type}
\sum_{\substack{(i,j)\in\lambda, \\ j-i = n-m}} r_{i,j}
= \sum_{\substack{(i,j)\in\lambda, \\ j-i = n-m}} b_{i,j}
= \sum_{i=1}^m \sum_{j=1}^n x_{i,j}
=: {\rm Rec}_{m,n}(x) \, .
\end{equation}
Let now $(m_1,n_1),\dots, (m_l,n_l)$ be the corners of a partition $\lambda$, ordered so that $m_1 > \dots > m_l$ and $n_1 < \dots < n_l$.
Then, \eqref{eq:RSK_Burge_type} holds for $(m,n)=(m_k,n_k)$ and, if $k>1$, also for $(m,n)=(m_k,n_{k-1})$ (both are border boxes by construction).
It is then clear that the `global sum' of the tableau $x$ can be expressed as a linear combination with integer coefficients of `rectangular sums'~\eqref{eq:RSK_Burge_type}; specifically, we have the representation
\[
\sum_{(i,j)\in \lambda} x_{i,j}
= {\rm Rec}_{m_1,n_1}(x) + \sum_{k=2}^l \left[{\rm Rec}_{m_k,n_k}(x) - {\rm Rec}_{m_k,n_{k-1}}(x)\right] \, .
\]
We thus deduce a fact crucial for our purposes: for any shape $\lambda$ with corners $(m_1,n_1),\dots, (m_l,n_l)$ as above, define $\{\omega_{i,j}\colon (i,j)\in \lambda\}$ by setting
\[
\omega_{i,j} :=
\begin{cases}
+1 &\text{if there exists $k$ such that } j-i = n_k - m_k \, , \\
-1 &\text{if there exists $k$ such that } j-i = n_{k-1} - m_k \, , \\
0 &\text{otherwise.}
\end{cases}
\]
We then have that
\begin{equation}
\label{eq:sumTableau}
\sum_{(i,j)\in \lambda} \omega_{i,j} r_{i,j}
= \sum_{(i,j)\in \lambda} x_{i,j}
= \sum_{(i,j)\in \lambda} \omega_{i,j} b_{i,j}
\end{equation}
for all $x\in\tab_{\Z_{\geq 0}}(\lambda)$, where $r:= \RSK(x)$ and $b:=\Burge(x)$.

\begin{ex}
\label{ex:RSK_Burge}
In Fig.~\ref{fig:RSK-Burge} we give a reference example of the RSK and Burge maps. 
\begin{figure}[t!]
\centering
{\includegraphics[width=25mm]{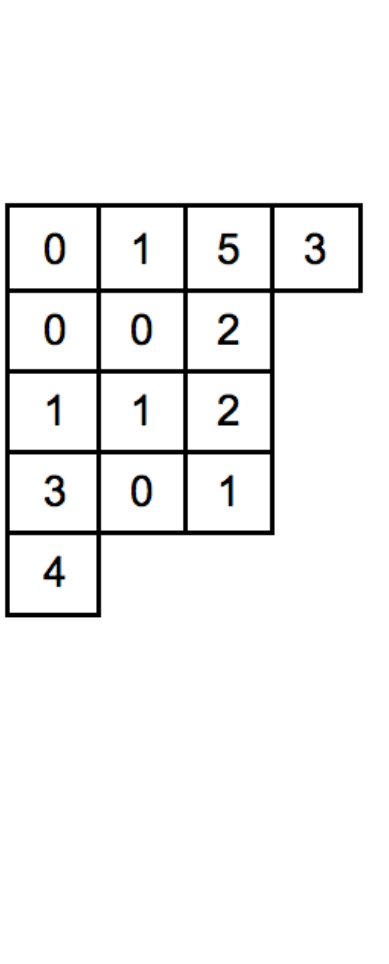}}
\hspace{0.5cm} \raisebox{100pt}{$\begin{matrix}\xmapsto{\ \RSK \ } \\ \\   \\  \\ \xmapsto{\ \Burge \ } \end{matrix}$} \hspace{0.7cm}
{\includegraphics[width=25mm]{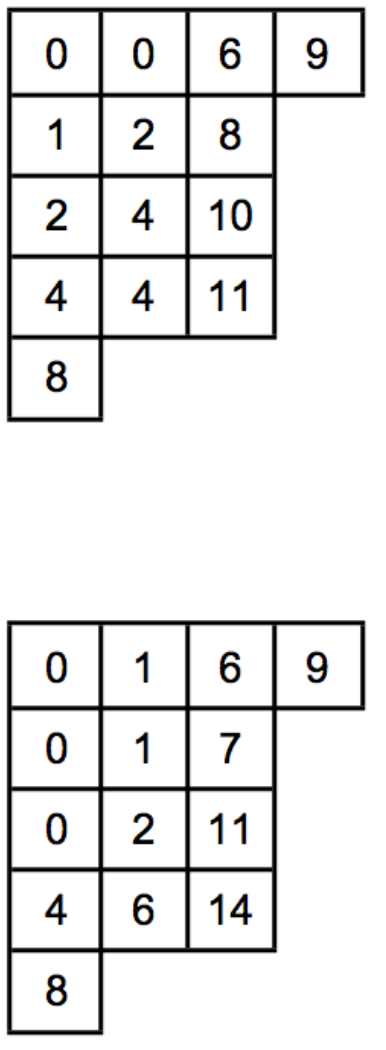}}
\caption{Illustration of the RSK and Burge correspondences.}
\label{fig:RSK-Burge}
\end{figure}
The input is a tableau $x\in \tab_{\Z_{\geq 0}}(\lambda)$ with $\lambda = (4,3,3,3,1)$.
The two outputs are the interlacing tableaux $r=\RSK(x)$ and $b=\Burge(x)\in\interlacingtab_{\Z_{\geq 0}}(\lambda)$.
One can easily verify the identities~\eqref{eq:LPP_RSK_Burge}-\eqref{eq:RSK_Burge_type}.
For instance, taking the box $(3,3)$ in the border strip $\Fro$ of $\lambda$, we have
\[
{\rm Rec}_{3,3}(x)=r_{3,3}+r_{2,2}+r_{1,1}=b_{3,3}+b_{2,2}+b_{1,1}=12 \, .
\]
For this shape, we have $\omega_{i,j} = 1$ when $(i,j)\in \{ (5,1), (4,3), (3,2), (2,1), (1,4) \}$; $\omega_{i,j}=-1$ when $(i,j)\in \{ (4,1), (1,3)\}$; and $\omega_{i,j}=0$ otherwise.
\end{ex}

\subsection{Equidistribution of random RSK and Burge tableaux}
\label{subsec:proofUpDownLPP}

We now formulate as a lemma the key identity in the proof of Theorem~\ref{thm:UpDownLPP}.
In a broad sense, we will say that a random variable $G$ is geometrically distributed (with \emph{support} $\Z_{\geq k}$, for some integer $k\geq 0$, and \emph{parameter} $p\in (0,1)$) if
\[
\P(G=m) = p(1-p)^{m-k}
\qquad
\text{for all $m\in \Z_{\geq k}$.}
\]

\begin{lem}
\label{lem:RSK=Burge}
If $X$ is a random tableau of shape $\lambda$ with i.i.d.\ geometric entries, then
\begin{equation}
\RSK(X)\overset{D}{=} \Burge(X) \, .
\label{eq:RSK=Burge}
\end{equation}
\end{lem}

\begin{proof}
Assume first that $X$ has i.i.d.\ geometric entries with support $\Z_{\geq 0}$ and any parameter $p\in (0,1)$.
Fix a tableau $t\in\interlacingtab_{\Znonneg}(\lambda)$ and let $y:=\RSK^{-1}(t)$ and $z:=\Burge^{-1}(t)$.
It then follows from~\eqref{eq:sumTableau} that
\[
\begin{split}
\P(\RSK(X)=t)
= \P(X=y)
&= p^{\abs{\lambda}} (1-p)^{\sum_{(i,j)\in\lambda} y_{i,j}}
= p^{\abs{\lambda}} (1-p)^{\sum_{(i,j)\in\lambda} \omega_{i,j} t_{i,j}} \\
&= p^{\abs{\lambda}} (1-p)^{\sum_{(i,j)\in\lambda} z_{i,j}}
= \P(X=z)
= \P(\Burge(X)=t) \, ,
\end{split}
\]
where $\abs{\lambda} := \sum_{i\geq 1} \lambda_i$ is the size of $\lambda$.
This proves that $\RSK(X)$ and $\Burge(X)$ are equal in distribution.

The proof in the case of tableaux with i.i.d.\ geometric entries with support in $\Z_{\geq k}$, $k\geq 0$, follows immediately from the following observation: if we shift all the entries of a tableau by a constant $k$, i.e.\ set $Y_{i,j} := X_{i,j} +k$, then from~\eqref{eq:RSK_Greene}-\eqref{eq:Burge_Greene} we have
\begin{align*}
\RSK(Y)_{i,j} &= \RSK(X)_{i,j} +(i+j-1)k \, , \\
\Burge(Y)_{i,j} &= \Burge(X)_{i,j} +(i+j-1)k \, . \qedhere
\end{align*}
\end{proof}

By combining this lemma with~\eqref{eq:LPP_RSK_Burge}, we derive the announced conclusion.
\begin{proof}[Proof of Theorem~\ref{thm:UpDownLPP}]
Fix a partition $\lambda$  with border strip $\Fro$.
Let $X$ be a random tableau of shape $\lambda$, and denote by $L$ and $L^*$ the corresponding LPP and dual LPP tableaux, respectively.
Let $\RSK(X)$ and $\Burge(X)$ be the (random) images of $X$ under the RSK and Burge correspondences, respectively.
By~\eqref{eq:LPP_RSK_Burge} we have the exact (not only distributional!) equalities 
\[
\RSK(X)_{m,n} = L_{m,n}\quad\text{and}\quad \Burge(X)_{m,n} = L^*_{m,n},\qquad\text{ for all $(m,n)\in\Fro$.}
\]

Assume first that the entries of $X$ are i.i.d.\ geometric variables, so that $X$ takes values in $\tab_{\Z_{\geq 0}}(\lambda)$.
By Lemma~\ref{lem:RSK=Burge}, $\RSK(X)$ and $\Burge(X)$ have the same joint distribution.
It follows that the restrictions of the LPP and dual LPP tableaux to the border strip, namely $\RSK(X)|_\Fro=L|_\Fro$ and $\Burge(X)|_\Fro=L^*|_\Fro$, are also equal in distribution.

Suppose now that $X$ has i.i.d.\ exponential entries of rate $\alpha$.
We have the convergence $\epsilon X^{(\epsilon)} \xrightarrow{\epsilon\downarrow 0} X$ in law, where $X^{(\epsilon)}$ is a random tableau with i.i.d.\ geometric entries with parameter $p=1-\e^{-\epsilon \alpha}$ (any support $\Z_{\geq k}$ works).
Denote by $L^{(\epsilon)}$ and $L^{*(\epsilon)}$ the LPP and dual LPP tableaux, respectively, corresponding to the input tableau $X^{(\epsilon)}$.
It is then immediate to see from the definition that $\epsilon L^{(\epsilon)}$ and $\epsilon L^{*(\epsilon)}$ are the LPP and dual LPP tableaux, respectively, corresponding to $\epsilon X^{(\epsilon)}$.
Since both the LPP and dual LPP tableaux are continuous functions of the input tableau, we deduce from the continuous mapping theorem (see~\cite[Theorem~3.2.10]{durrett19}) that
\[
\epsilon L^{(\epsilon)} \xrightarrow{\epsilon\downarrow 0} L
\qquad\quad \text{and} \qquad\quad
\epsilon L^{*(\epsilon)} \xrightarrow{\epsilon\downarrow 0} L^*
\]
in law.
As the claim has already been proven for geometric weights, we know that $L^{(\epsilon)}|_\Fro \overset{D}{=} L^{*(\epsilon)}|_\Fro$.
It follows that $L|_\Fro \overset{D}{=} L^*|_\Fro$, as required.
\end{proof}

\begin{rem}
It is possible to extend Theorem~\ref{thm:RSK} to view the RSK and Burge correspondences as acting on tableaux with \emph{real}, instead of integer, entries; see~\cite[\S~2]{bisiOConnellZygouras20} for the construction.
Viewed as real functions, these bijections turn out to be volume-preserving (i.e.\ their Jacobians are both of modulus 1 almost everywhere).
Using this property, the argument used to prove Lemma~\ref{lem:RSK=Burge} can then be adapted to establish the distributional equality between $\RSK(X)$ and $\Burge(X)$ also when the input tableau $X$ has exponential i.i.d.\ entries.
The proof of Theorem~\eqref{thm:UpDownLPP} in the exponential case would then be akin to the geometric case, with no need to take a scaling limit.
\end{rem}

\begin{rem}
Let $X$ be a random tableau of shape $\lambda$.
The proof of Lemma~\ref{lem:RSK=Burge} suggests a \emph{sufficient} condition on the joint distribution of $X$ in order for~\eqref{eq:RSK=Burge} (and, hence, Theorem~\ref{thm:UpDownLPP}) to hold.
Such a condition is the property that the $\P(X=y) = \P(X=z)$ whenever $y,z\in \tab_{\Z_{\geq 0}}(\lambda)$ have equal global sum, i.e.\ $\sum_{(i,j)\in\lambda} y_{i,j} = \sum_{(i,j)\in\lambda} z_{i,j}$.
If we further assume the entries of $X$ to be independent, this property forces the entries of $X$ to be i.i.d.\ with a geometric distribution.
The latter claim follows from the fact that, if $f,g_1,\dots,g_k$ are probability mass functions on $\Z_{\geq 0}$ such that $g_1(x_1)g_2(x_2) \cdots g_k(x_k)$ is proportional to $f(x_1+\dots +x_k)$ for all $x_1,\dots,x_k\in \Z_{\geq 0}$, then $f,g_1,\dots,g_k$ are necessarily all geometric with the same parameter.
\end{rem}

\section{From a probabilistic to a combinatorial conjecture}
\label{sec:comb-iden}

In this section we reformulate Conjecture~\ref{main-conj} by showing its equivalence to Conjecture~\ref{main-conj-reformulated}.
We start by discussing the two families of combinatorial objects and defining the relevant associated quantities appearing in identity~\eqref{eq:comb-iden}.

\subsection{Staircase shape Young tableaux}
\label{subsec:staircaseTableaux}
Let $\staircase_n$ denote the partition $(n-1,n-2,\ldots,1)$ of $N=n(n-1)/2$; as a Young diagram we will refer to $\staircase_n$ as the \firstmention{staircase shape of order~$n$}.
Let $\syt(\staircase_n)$ denote the set of standard Young tableaux of shape $\staircase_n$.
We associate with each $t\in\syt(\staircase_n)$ several parameters, which we denote by $\diag_t$, $\sigma_t$, $\deg_t$, and $f_t$.
(Note: these definitions are somewhat technical; refer to Example~\ref{ex:syt6} below for a concrete illustration that makes them easier to follow.)

First, we define
\[
\diag_t := (t_{n-1,1}, t_{n-2,2}, \ldots, t_{1,n-1})
\]
to be the vector of corner entries of~$t$ read from bottom-left to top-right. 
Second, we define $\sigma_t\in \sym_{n-1}$ to be the permutation  encoding the ordering of the entries of $\diag_t$, so that $\diag_t(j) < \diag_t(k)$ if and only if $\sigma_t(j) < \sigma_t(k)$ for all $j,k$.
The vector $\diagsorted_t$ will denote the increasing rearrangement of $\diag_t$, so that $\diagsorted_t(k):=\diag_t(\sigma_t^{-1}(k))$ for all~$k$.
For later convenience we also adopt the notational convention that  $\diagsorted_t(0)=0$.

Notice that a tableau $t \in \syt(\staircase_n)$ encodes a growing sequence
\begin{equation}
\label{eq:diagseq}
\emptyset = \lambda^{(0)} \nearrow \lambda^{(1)}
\nearrow \lambda^{(2)} \nearrow \ldots \nearrow \lambda^{(N)}=
\staircase_n
\end{equation}
of Young diagrams that starts from the empty diagram, ends at $\staircase_n$, and such that each $\lambda^{(k)}$ is obtained from $\lambda^{(k-1)}$ by adding the box $(i,j)$ for which $t_{i,j}=k$.
We then define the vector $\deg_t = (\deg_t(0),\ldots,\deg_t(N-1))$, where $\deg_t(k)$ is the number of boxes $(i,j)\in \staircase_n / \lambda^{(k)}$ such that $\lambda^{(k)} \cup \{(i,j)\}$ is a Young sub-diagram of $\staircase_n$.
We may interpret $\deg_t(k)$ as the out-degree of $\lambda^{(k)}$ regarded as a vertex of the directed graph $\young(\staircase_n)$ of Young diagrams contained in $\staircase_n$ (a sublattice of the \firstmention{Young graph}, or \firstmention{Young lattice}, $\mathcal{Y}$), with edges corresponding to the box-addition relation $\mu \nearrow \lambda$; see Fig.~\ref{subfig:youngGraph}.

Notice that the randomly growing Young diagram model introduced in Subsection~\ref{subsec:models} is nothing but a continuous-time simple random walk on $\young(\staircase_n)$ that starts from the empty diagram (and necessarily ends at $\staircase_n$).
Let $T$ be the (random) standard Young tableau that encodes the path of such a random walk, i.e.\ the associated sequence of growing diagrams~\eqref{eq:diagseq}; then,
\begin{equation}
\label{eq:deg_t}
\P(\nobreak{T=t}) = \prod_{j=0}^{N-1} \frac{1}{\deg_t(j)}
\qquad
\text{for all } t\in\syt(\staircase_n) \, .
\end{equation}

Finally, we define the \firstmention{generating factor} of~$t$ as the rational function
\begin{equation}
f_t(x_1,\ldots,x_{n-1}) :=
\prod_{k=1}^{n-1} \ \ \prod_{\diagsorted_t(k-1) < j \le \diagsorted_t(k)} \frac{1}{x_k + \deg_t(j)} \, .
\label{eq:f_t}
\end{equation}
Recall from Section~\ref{sec:intro} that the vector $\bm{V}_n$ records the times when the corner boxes of the shape $\delta_n$ are added in the randomly growing Young diagram model / random walk on $\young(\staircase_n)$.
The generating factor $f_t(x_1,\dots,x_{n-1})$  is, essentially, the joint Fourier transform of the vector $\bm{V}_n$, conditioned on the random walk path encoded by the tableau $t$; see Subsection~\ref{subsec:equivalenceConjectures}.

\definecolor{colorone}{rgb}{0,0.2,0.8}
\definecolor{colortwo}{rgb}{1,0,0}
\definecolor{colorthree}{rgb}{0,0.4,0}
\definecolor{colorfour}{rgb}{0.5,0,0.5}
\definecolor{colorfive}{rgb}{0,0,0}

\begin{ex}
\label{ex:syt6}
For the tableau $t$ shown in Fig.~\ref{fig:syt-sn6} (left), we have 
\begin{gather*}
\begin{aligned}
\diag_t &= 
({{\color{colorone}10}, {\color{colorthree}13}, {\color{colorfive}15}, {\color{colorfour}14}, {\color{colortwo}11}}),\\
\sigma_t &= 
({{\color{colorone}1}, {\color{colorthree}3}, {\color{colorfive}5}, {\color{colorfour}4}, {\color{colortwo}2}}), \\
\deg_t &=
({{\color{colorone}1}, {\color{colorone}2}, {\color{colorone}2}, {\color{colorone}3}, {\color{colorone}3}, {\color{colorone}3}, {\color{colorone}4}, {\color{colorone}4}, {\color{colorone}4}, {\color{colorone}4}, {\color{colortwo}3}, {\color{colorthree}2}, {\color{colorthree}3}, {\color{colorfour}2}, {\color{colorfive}1}}),
\end{aligned} \\
f_t =
{\color{colorone}\frac{1}{(x_{1}+{1})(x_{1}+{2})^2 (x_{1}+{3})^3 (x_{1}+{4})^4}} \cdot
{\color{colortwo}\frac{1}{x_{2}+{3}}} \cdot
{\color{colorthree}\frac{1}{(x_{3}+{2})(x_{3}+{3})}} \cdot
{\color{colorfour}\frac{1}{x_{4}+{2}}} \cdot
{\color{colorfive}\frac{1}{x_{5}+{1}}} \, .
\end{gather*}
For example, $\deg_t(5)=3$, because $\lambda^{(5)}$, the sixth Young diagram in the growth sequence associated with the tableau $t$, is the partition $(3,1,1)$, which has $3$ external corners lying within $\delta_5$, that is, its out-degree in the graph $\young(\delta_5)$ is $3$.

Here, we have used colors to illustrate how the entries of $\diag_t$ determine a decomposition of $\deg_t$ into blocks, which correspond to different variables $x_k$ in the definition of the generating factor $f_t$.
\end{ex}

\begin{figure}[t!]
\centering
\centering
{\includegraphics[height=30mm]{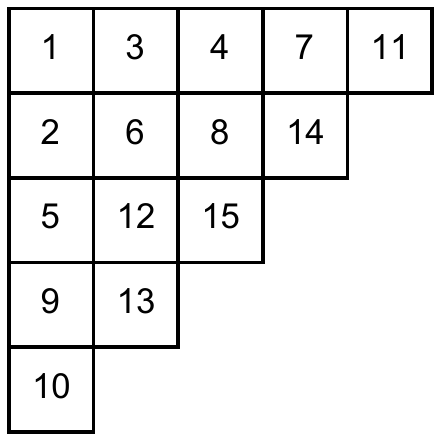}}
\hspace{0.3cm} \raisebox{40pt}{$\xmapsto{\ \eg\ }$} \hspace{0.3cm}
{\includegraphics[height=30mm]{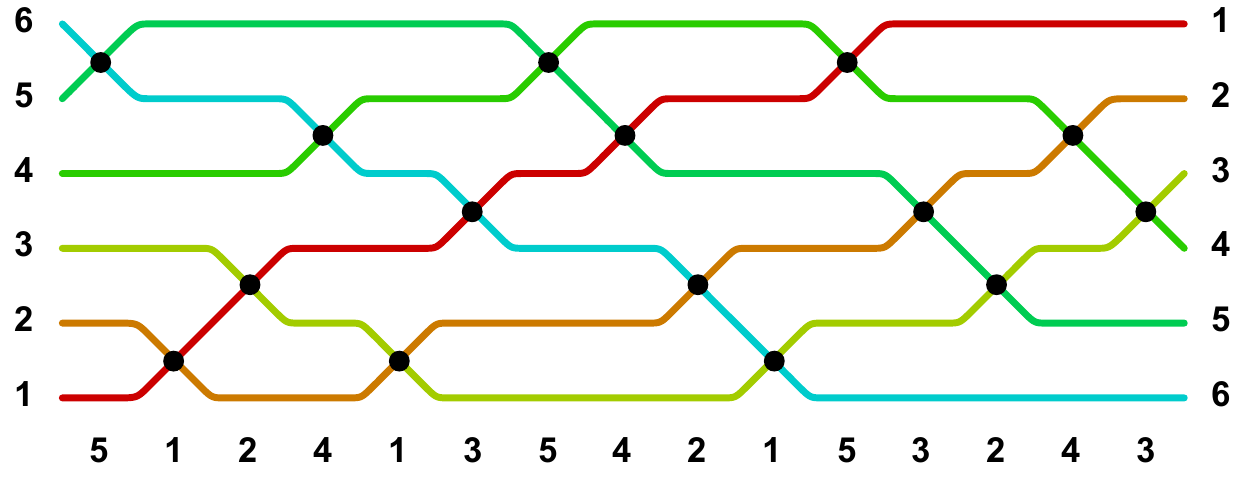}}
\caption{A staircase shape standard Young tableau $t$ of order~6, shown in `English notation', and the associated sorting network $s=\eg(t)$ of order~6 (illustrated graphically as a wiring diagram) with swap sequence $(5,1,2,4,1,3,5,4,2,1,5,3,2,4,3)$.}
\label{fig:syt-sn6}
\end{figure}

\subsection{Sorting networks}
\label{subsec:sortingNetworks}

Recall that a \firstmention{sorting network of order~$n$} is a synonym for a reduced word decomposition of the reverse permutation $\rev_n = (n,n-1,\ldots,1)$ in terms of the Coxeter generators $\tau_j = (j \ \ j+1)$, $1\le j< n $, of the symmetric group $S_n$.
Formally, a sorting network is a sequence of indices $s=(s_1,\ldots,s_{N})$ of length $N=n(n-1)/2$, such that
$ 1\le s_j <n $ for all $j$ and 
$ \rev_n = \tau_{s_{N}} \cdots \tau_{s_2} \tau_{s_1}$.

We denote by $\sort_n$ the set of sorting networks of order~$n$.
The elements of $\sort_n$ can be portrayed graphically using \firstmention{wiring diagrams}, as illustrated in Fig.~\ref{fig:syt-sn6}.
They can also be interpreted as \firstmention{maximal length chains} in the \firstmention{weak  Bruhat order} or, equivalently, shortest paths in the poset lattice (which is the Cayley graph of $\sym_n$ with the adjacent transpositions $\tau_j$ as generators, see Fig.~\ref{subfig:caleyGraph}) connecting the identity permutation $\id_n$ to the permutation $\rev_n$.
We refer to~\cite{bjoernerBrenti05, humphreys90} for details on this terminology.

We associate with a sorting network $s \in \sort_n$ the parameters $\fin_s$, $\pi_s$, $\deg_s$, and $g_s$ that will play a role analogous to the parameters $\diag_t$, $\sigma_t$, $\deg_t$, and $f_t$ for $t\in \syt(\staircase_n)$.

We define the vector $ \fin_s = (\fin_s(1),\fin_s(2),\ldots,\fin_s(n-1))$ by setting
\[
\fin_s(k) := \max\{ 1\le j \le N\colon s_j = k \}
\]
to be the index of the last swap occurring between positions $k$ and $k+1$.
We define $\pi_s \in \sym_{n-1}$ to be the permutation encoding the ordering of the entries of $\fin_s$, so that $\fin_s(j) < \fin_s(k)$ if and only if $\pi_s(j) < \pi_s(k)$. 
We denote by $\finsorted_s$ the increasing rearrangement of $\fin_s$, and use the notational convention $\finsorted_s(0)=0$.

We next define $\deg_s = (\deg_s(0),\ldots,\deg_s(N-1))$ to be the vector with coordinates
$\deg_s(k) := |\{ 1\le j\le n-1\colon \nu^{(k)}(j)<\nu^{(k)}(j+1)\}|$,
where $\nu^{(k)} := \tau_{s_k}\cdots \tau_{s_2}\tau_{s_1}$ is the $k$-th permutation in the path encoded by $s$.
In  words, $\deg_s(k)$ is the out-degree of $\nu^{(k)}$ in the Cayley graph of $S_n$ (with the adjacent transpositions as generators); see Fig.~\ref{subfig:caleyGraph}.

Notice that the oriented swap process on $n$ particles introduced in Subsection~\ref{subsec:models} is a continuous-time simple random walk on this graph that starts from $\id_n$ (and necessarily ends at $\rev_n$).
The (random) sorting network $S$ that encodes the path of the OSP is then distributed as follows:
\begin{equation}
\label{eq:deg_s}
\P(\nobreak{S=s}) = \prod_{j=0}^{N-1} \frac{1}{\deg_s(j)}
\qquad
\text{for all } s\in\sort_n \, .
\end{equation}

Finally, the \firstmention{generating factor} $g_s$ of $s$ is defined, analogously to~\eqref{eq:f_t}, as the rational function
\begin{equation}
g_s(x_1,\ldots,x_{n-1}) = \prod_{k=1}^{n-1} \ \ \prod_{\finsorted_s(k-1) < j \le \finsorted_s(k)} \frac{1}{x_k + \deg_s(j)} \, .
\label{eq:g_s}
\end{equation}
Recall from Section~\ref{sec:intro} that the vector $\bm{U}_n$ records the times when the last swap between particles in any two neighboring positions occurs in the oriented swap process / random walk on the graph defined above.
The generating factor $g_s(x_1,\dots,x_{n-1})$ is, essentially, the joint Fourier transform of the vector $\bm{U}_n$, conditioned on the random walk path encoded by the sorting network $s$; see Subsection~\ref{subsec:equivalenceConjectures}.

\begin{ex} 
\label{ex:sortingnet6}
For the sorting network $s = (5, 1, 2, 4, 1, 3, 5, 4, 2, 1, 5, 3, 2, 4, 3) \in \sort_6$ shown in Fig.~\ref{fig:syt-sn6} (right), we have that
\begin{gather*}
\begin{aligned}
\fin_s &
= ({{\color{colorone}10}, {\color{colorthree}13}, {\color{colorfive}15}, {\color{colorfour}14}, {\color{colortwo}11}}),\\
\pi_s &= ({{\color{colorone}1}, {\color{colorthree}3}, {\color{colorfive}5}, {\color{colorfour}4}, {\color{colortwo}2}}),\\
\deg_s
 &= ({{\color{colorone}5}, {\color{colorone}4}, {\color{colorone}3}, {\color{colorone}3}, {\color{colorone}3}, {\color{colorone}2}, {\color{colorone}3}, {\color{colorone}2}, {\color{colorone}2}, {\color{colorone}3}, {\color{colortwo}2}, {\color{colorthree}1}, {\color{colorthree}2}, {\color{colorfour}1}, {\color{colorfive}1}}),
 \end{aligned} \\
g_s = 
{\color{colorone}\frac{1}{(x_{1}+{5})(x_{1}+{4})(x_{1}+{3})^5(x_{1}+{2})^3}} \cdot
{\color{colortwo}\frac{1}{x_{2}+{2}}} \cdot
{\color{colorthree}\frac{1}{(x_{3}+{1})(x_{3}+{2})}} \cdot
{\color{colorfour}\frac{1}{x_{4}+{1}}} \cdot
{\color{colorfive}\frac{1}{x_{5}+{1}}} \, .
 \end{gather*}
The above parameters are shown using color coding as in Example~\ref{ex:syt6}.
\end{ex}

\subsection{The Edelman-Greene correspondence}
\label{subsec:EG}

Stanley conjectured and then proved~\cite{stanley84} that sorting networks are equinumerous with staircase shape Young tableaux of the same order, i.e.\ $|\sort_n|=|\syt(\staircase_n)|$.
Edelman and Greene~\cite{edelmanGreene87} found an explicit combinatorial bijection $\eg:\syt(\staircase_n)\to\sort_n$, which is now known as the Edelman--Greene correspondence (see also~\cite{hamakerYoung14, lascouxSchutzenberger82, little03}).
The standard tableau and the sorting network of Examples~\ref{ex:syt6} and~\ref{ex:sortingnet6} (see also Fig.~\eqref{fig:syt-sn6}) are associated to each other via $\eg$.

The map $\eg$ can be conveniently described in terms of the Sch\"utzenberger operator iterated $N$ times until all the original labels of a tableau $t\in \syt(\staircase_n)$ are `evacuated' (recall that $N=n(n-1)/2$ is the number of boxes of the Young diagram $\delta_n$).

The Sch\"utzenberger operator $\Phi \colon \syt(\staircase_n)\to \syt(\staircase_n)$ acts as follows.
For a tableau $t = (t_{i,j})\in \syt(\staircase_n)$, define the \emph{evacuation path} to be the sequence $c=(c_1,c_2,\ldots, c_{n-1})$ of boxes $c_m=(i_m,j_m)\in\delta_n$ such that:
\begin{enumerate}[label=(\roman*)]
\item $c_1=(i_1,j_1)$ where  $t_{i_1,j_1}=N$;
\item if $c_{m-1}=(a,b)$, then $c_m$ is the box $(a-1,b)$ if $t_{a-1,b}>t_{a,b-1}$ and the box $(a,b-1)$ otherwise, for all $2\leq m \leq n-1$.
\end{enumerate}
In this definition the convention is that $t_{i,0}=t_{j,0}=0$ for all $i$ and $j$.
Note that $c_{n-1}=(1,1)$.
Define $t' = (t'_{i,j})_{(i,j)\in \delta_n}$ by letting $t'_{c_m}:=t_{c_{m+1}}$ for $m=1,\dots n-2$, $t'_{c_{n-1}}:=0$, and $t'_{i,j}:=t_{i,j}$ whenever $(i,j)\notin c$ (\emph{sliding} along the evacuation path).
Then, the tableau $\Phi(t)=(\hat{t}_{i,j})_{(i,j)\in\delta_n}$ is constructed by setting $\hat{t}_{i,j}=t'_{i,j}+1$ for all $(i,j)\in \delta_n$ (\emph{increment}).

\begin{figure}[t!]
\centering
{\includegraphics[width=.85\columnwidth]{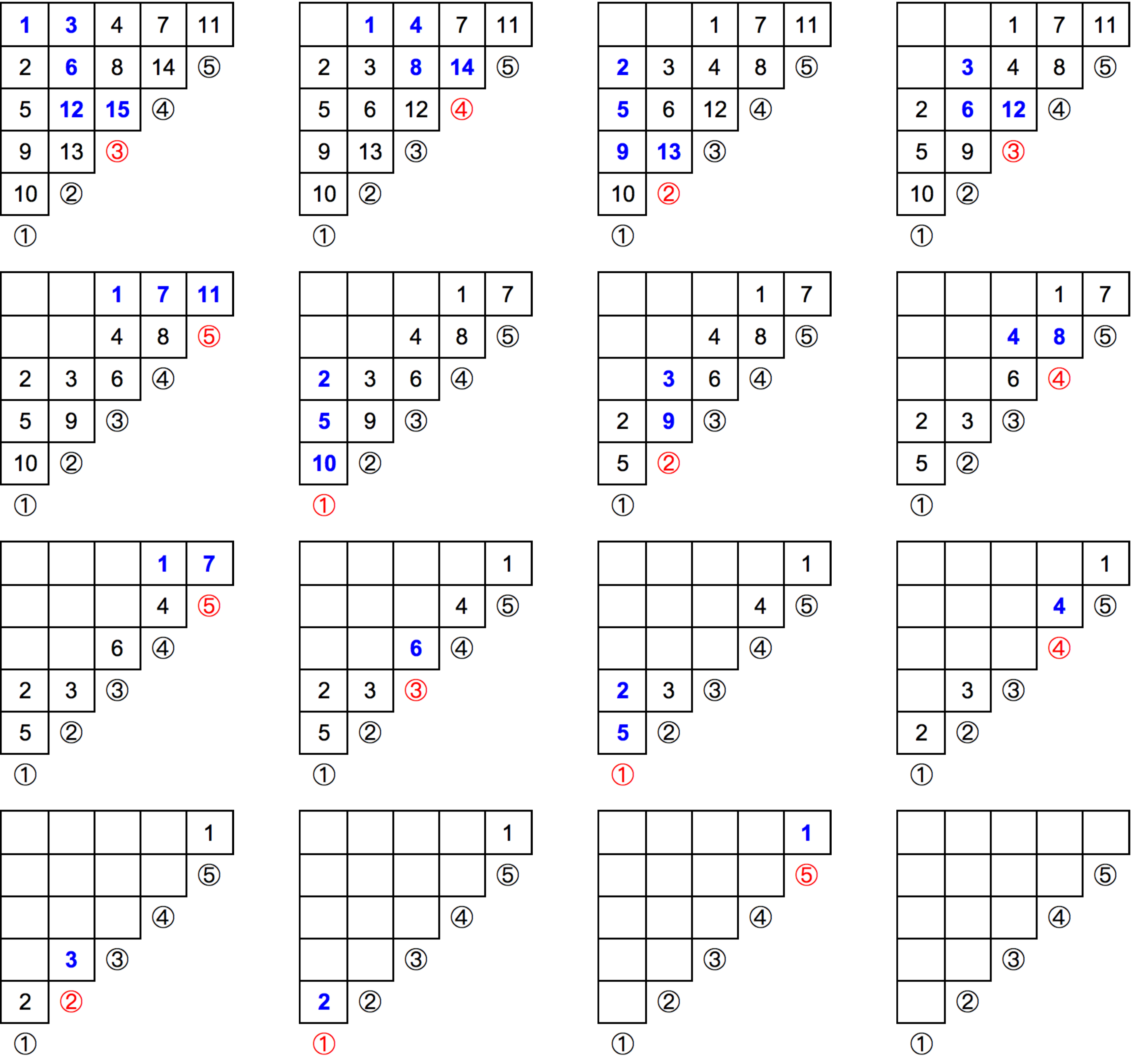}}
\caption{The map $\eg:\syt(\staircase_n)\to\sort_n$ can be visualized as `emptying' the tableau $t$.
Here the tableau is the same as in Fig.~\ref{fig:syt-sn6}.
We highlight in blue the evacuation paths (obtained by starting from the maximum entry and repeatedly moving to the box above or to the left that contains the largest entry).
At each step we perform an outward sliding along the evacuation path, keeping track of $j_{\max}$ (in red).
To keep the picture as intuitive as possible, we do not perform the increment $+1$ (the omission of this step does not change the sequence of $j_{\max}$'s) and we only indicate the original labels of the tableau $t$.
The associated sorting network is the sequence of indices $j_{\max}$'s read in reverse order: $(5,1,2,4,1,3,5,4,2,1,5,3,2,4,3)$.}
\label{fig:syt-sn6_emptying}
\end{figure}

In the notation of~\cite[\S~4]{angelHolroydRomikVirag07}, for a tableau $t\in\syt(\staircase_n)$, set $j_{\max}(t) := j_1$.
Then, the Edelman--Green map takes the tableau $t$ as an input and returns the sorting network
\[
\eg(t) := \left(j_{\max} \left(\Phi^{N-m}(t)\right)\right)_{1\leq m\leq N} \, ,
\]
where $\Phi^{m}$ denotes the $m$-th iterate of $\Phi$.
See Fig.~\ref{fig:syt-sn6_emptying}.

The following result is easy to guess from Examples~\ref{ex:syt6} and~\ref{ex:sortingnet6}.

\begin{prop}
\label{prop:edelman-greene-params}
If $t\in \syt_n$ and $s=\eg(t) \in \sort_n$, then 
\begin{equation}
\label{eq:edelman-greene-params}
\fin_s = \diag_t \qquad\text{and}\qquad
\pi_s = \sigma_t \, .
\end{equation}
\end{prop}
\begin{proof}
The second relation follows trivially from the first.
This first identity is an easy consequence of the definition of the Edelman--Greene correspondence, and specifically of the way the map $\eg:\syt(\staircase_n)\to\sort_n$ can be visualized as `emptying' the tableau $t$ (see the discussion above and Fig.~\ref{fig:syt-sn6_emptying}) by repeatedly applying the Sch\"utzenberger operator:
\begin{align*}
\fin_s(k) 
&= \max \{ 1\le m\le N \,:\, j_\textrm{max}(\Phi^{N-m}(t)) = k \} \\
&= N - \min \{ 0\le r\le N-1\,:\, j_\textrm{max}(\Phi^r(t)) = k \} \\
&= N- (N- t_{n-k,k}) = t_{n-k,k} = \diag_t(k) \, .
\qedhere
\end{align*}
\end{proof}

\subsection{The combinatorial identity}
\label{subsec:combinIdentity}

Let $\freevecspace_{n-1}$ denote the free vector space generated by the elements of $\sym_{n-1}$ over the field of rational functions $\C_\mathbf{x}^{n-1}:=\C(x_1,\ldots,x_{n-1})$.
Define the following generating functions as elements of $\freevecspace_{n-1}$:
\begin{align}
F_n(x_1,\ldots,x_{n-1}) &:=\!\!\! \sum_{t\in \syt(\staircase_n)} f_t(x_1,\ldots,x_{n-1}) \sigma_t \, ,
\label{eq:genfun-F}
\\
G_n(x_1,\ldots,x_{n-1}) &:= \sum_{s\in \sort_n} 
g_s(x_1,\ldots,x_{n-1}) \pi_s \, .
\label{eq:genfun-G}
\end{align}
Conjecture~\ref{main-conj-reformulated} is the identity $F_n(x_1,\ldots,x_{n-1}) = G_n(x_1,\ldots,x_{n-1})$ (an equality of vectors with $(n-1)!$ components).
\begin{rem}
Note that in general it is \emph{not} true that $f_t = g_s$ if $s=\eg(t)$, as Examples~\ref{ex:syt6} and~\ref{ex:sortingnet6} clearly show. Thus, the Edelman--Greene correspondence does not seem to imply the conjecture in an obvious way. However, using~\eqref{eq:edelman-greene-params} we see that the  correspondence does imply the limiting case 
\begin{equation}
\lim\limits_{x \to\infty} x^N (F_n(x,\ldots,x)-G_n(x,\dots,x))=0 \, .
\label{eq:genfun-G_limiting}
\end{equation}
The above limit is equivalent to the statement
\[
\left|\{t\in\syt(\staircase_n)\,\colon\, \sigma_t=\gamma\}\right|=\left|\{s\in\sort_n\,\colon\, \pi_s=\gamma\}\right| \qquad \text{for all $\gamma\in S_{n-1}$} \, ,
\]
which is true by Proposition~\ref{prop:edelman-greene-params}.
\end{rem}
\begin{rem}
It is natural to wonder if there exists a bijection $\phi\colon \syt(\staircase_n) \to \sort_n$ (necessarily different from $\eg$), such that $f_t = g_{\phi(t)}$ for all $t\in \syt(\staircase_n)$, thus leading to a proof of Conjecture~\ref{main-conj}.
However, already for $n=4$, one can verify using Fig.~\ref{fig:syt-sn4} that the two sets of generating factors $\{f_t\}_{t\in \syt(\staircase_n)}$ and $\{g_s\}_{s\in \sort_n}$ are different.
Therefore, no bijection between $\syt(\staircase_n)$ and $\sort_n$ has the desired property.
\end{rem}
The calculation of $F_n(x_1,\ldots,x_{n-1})$ and $G_n(x_1,\ldots,x_{n-1})$ involves a summation over $|\syt(\staircase_n)|=|\sort_n|=N!/(1^{n-1}\cdot3^{n-2}\cdots(2n-3)^1)$ elements (e.g, $768$ elements for $n=5$ and $292864$ elements for $n=6$). For $n\leq 6$ this calculation is feasible by using symbolic algebra software.
We wrote code in Mathematica --- downloadable as a companion package \cite{orientedswaps-mathematica} to this paper --- to perform this calculation and check that the two functions are equal, thus proving Theorem~\ref{thm:main-thm2}.

\begin{figure}[t!]
\centering
{\includegraphics[width=1\columnwidth]{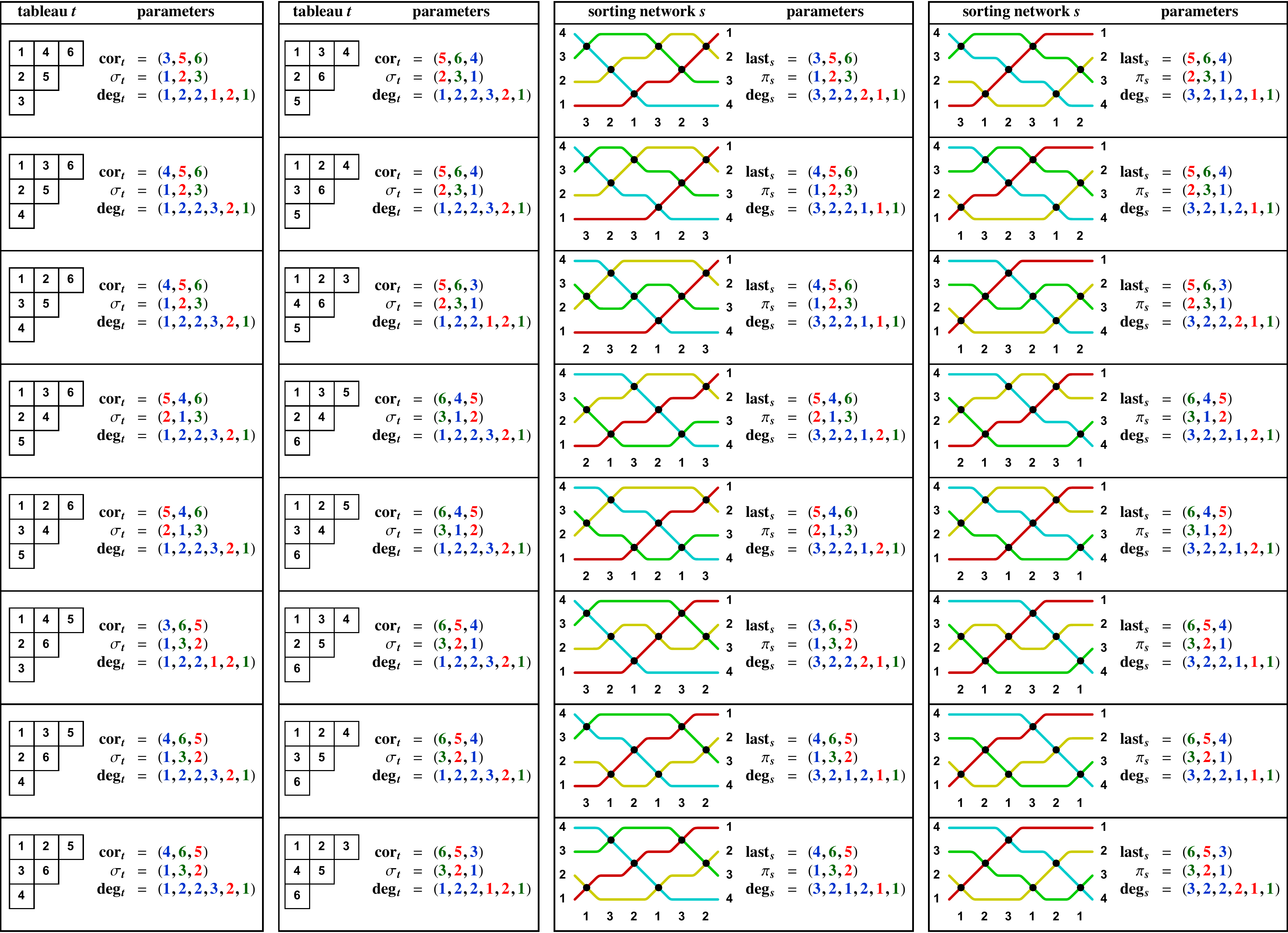}}
\caption{The $16$ staircase shape standard Young tableaux and sorting networks of order $4$ (ordered so that entries in the same relative positions in the two tables correspond to each others via the Edelman--Greene correspondence). As in Examples~\ref{ex:syt6}-\ref{ex:sortingnet6}, the coloring of the parameter entries emphasizes how different entries of $\deg_t$ and $\deg_s$ correspond to different factors in the definition of the generating factors $f_t$ and $g_s$.}
\label{fig:syt-sn4}
\end{figure}

\begin{ex}
For $n=4$, the generating functions can be computed by hand using the tables shown in Fig.~\ref{fig:syt-sn4} above.
For example, 
the component of the two generating functions associated with the identity permutation $\operatorname{id}=(1,2,3)$ is
\[
\begin{gathered}
\big(F_4(x_1,x_2,x_3)\big)_{\id}
= \big(G_4(x_1,x_2,x_3)\big)_{\id} \\
= 
\frac{x_1 + 2x_2 + 5}{(x_1+1) (x_1+2)^2 (x_1+3) (x_2+1) (x_2+2) (x_3+1)} \, .
\end{gathered}
\]
\end{ex}

\subsection{Equivalence of combinatorial and probabilistic conjectures}
\label{subsec:equivalenceConjectures}

We now prove the equivalence between Conjectures~\ref{main-conj} and~\ref{main-conj-reformulated}.
Conjecture~\ref{main-conj} can be viewed as claiming the equality $p_{\bm{U}_n} = p_{\bm{V}_n}$ of the joint density functions of $\bm{U}_n$ and $\bm{V}_n$.
We thus aim to derive explicit formulas for $p_{\bm{U}_n}$ and $p_{\bm{V}_n}$.

\subsubsection*{Decomposition of the densities}

As discussed in Subsections~\ref{subsec:staircaseTableaux} and~\ref{subsec:sortingNetworks}, both the randomly growing Young diagram model and the oriented swap process can be interpreted as continuous-time random walks.
The idea is then to write the density function of the last swap times $\bm{U}_n$ (resp.\ $\bm{V}_n$) as a weighted average of the conditional densities conditioned on the path that the process takes to get from the initial state $\id_n$  (resp.\   $\emptyset$) to the final state $\rev_n$ (resp.\  $\staircase_n$):
\begin{equation}
\label{eq:decomposition}
\begin{split}
 p_{\bm{U}_n}(u_1,\ldots,u_{n-1}) &= \sum_{s \in \sort_n} \P(S=s)  \, p_{\bm{U}_n|S=s}(u_1,\ldots,u_{n-1}) \, ,
 \\
p_{\bm{V}_n}(v_1,\ldots,v_{n-1})&= \sum_{t \in \syt(\staircase_n)} \P(T=t) \, p_{\bm{V}_n|T=t}(v_1,\ldots,v_{n-1}) \, .
\end{split}
\end{equation}
Here, $s$ (resp.\ $t$) can be viewed as a realization of a simple random walk $S$ (resp.\  $T$) on  the Cayley graph of $S_n$ (resp.\ on the directed graph $\young(\staircase_n)$).
The probabilities $\P(S=s)$ and $\P(T=t)$ are simply given by~\eqref{eq:deg_t} and~\eqref{eq:deg_s}.
We will now deal with the conditional densities.

\subsubsection*{Conditional densities}
We will now show that the conditional densities $p_{\bm{U}_n|S=s}(u_1,\ldots,u_{n-1})$ and $p_{\bm{V}_n|T=t}(v_1,\ldots,v_{n-1})$ are completely determined by the vectors $\finsorted_s$ and $\diagsorted_t$ and their corresponding orderings $\sigma_t$ and $\pi_s$ in the simple random walks, and the sequences of out-degrees $\deg_t$ and $\deg_s$ along the paths (which correspond to the exponential clock rates to leave each vertex in the graph where the random walk is taking place).

In the case of the OSP conditioned on the path $S=s$, take a sequence of independent random variables $\xi_1,\dots,\xi_N$, where $\xi_j$ has exponential distribution with rate $\deg_s(j)$.
Once the OSP has reached the state 
$ \tau_{s_k}\cdots \tau_{s_2}\tau_{s_1}$,
there are $\deg_s(j)$ Poisson clocks running in parallel, so, by standard properties of Poisson clocks (see~\cite[Ex.~4.1, p.~264]{romik15}) the time until a swap occurs is distributed as $\xi_j$ and is independent of the choice of the swap actually occurring.
Let then $\eta _t$ be defined~as
\[
\eta_t :=
\begin{cases}
\id_n
&\text{if $0\leq t< \xi_1$,}\\
\tau_{s_1}
&\text{if $\xi_1\leq t< \xi_1+\xi_2$,}\\
\tau_{s_2} \tau_{s_1}
&\text{if $\xi_1+\xi_2\leq t< \xi_1+\xi_2+\xi_3$,}\\
\vdots&\vdots\\
\tau_{s_{N-1}} \cdots \tau_{s_2} \tau_{s_1}
&\text{if $\xi_1+\xi_2+\cdots+\xi_{N-1}\leq t< \xi_1+\xi_2+\cdots+\xi_N$,}\\
\tau_{s_N}\tau_{s_{N-1}} \cdots \tau_{s_2} \tau_{s_1}
&\text{if $\xi_1+\xi_2+\cdots+\xi_{N}\leq t$.}
\end{cases}
\]
Thanks to the remarks above, this construction gives the correct distribution for the process $\left(\eta _t\right)_{t\geq0}$ as an oriented swap process on $n$ particles.

Next, observe that the conditional density $p_{\bm{U}_n|S=s}(u_1,\ldots,u_{n-1})$ is nonzero on one and only one of the $(n-1)!$ Weyl chambers
\begin{equation}
\label{eq:WeylChamber}
W_{\gamma} := \left\{\bm{u} = (u_1,\dots, u_{n-1}) \in\mathbb{R}^{n-1}_{\geq 0} \colon u_{\gamma^{-1}(1)}\leq u_{\gamma^{-1}(2)}\leq \cdots \leq u_{\gamma^{-1}(n-1)}\right\}
\end{equation}
associated to each of the different possible orderings $\gamma\in S_{n-1}$ of the variables $u_1,\ldots,u_{n-1}$.
For a path $s\in \sort_n$, the permutation $\pi_s\in S_{n-1}$ encodes the information about the relative order of the variables $U_n(1), U_n(2),\ldots, U_n(n-1)$, hence the conditional density will be nonzero precisely on the chamber $W_{\pi_s}$.

The last piece of information needed to compute the conditional density is the vector of integers $\fin_s$ that encodes, for each $k$, the point along the path wherein the last swap between positions $k$ and $k+1$ occurred.
Denote by $\bmUsorted_n$ the increasing rearrangement of  $\bm{U}_n$, so that $\Usorted_n(1) \leq  \Usorted_n(2) \leq \ldots  \leq \Usorted_n(n-1)$ are the order statistics of $\bm{U}_n$.
Conditioned on $S=s$, we have that $\Usorted_n(k) = U_n(\pi_s^{-1}(k))$ and
\begin{align*}
\Usorted_n(1)&=\xi_1+\cdots+\xi_{\finsorted_s(1)} \, ,\\
\Usorted_n(2)-\Usorted_n(1)&=\xi_{\finsorted_s(1)+1}+\cdots+\xi_{\finsorted_s(2)} \, ,\\
&\,\,\,\vdots\\
\Usorted_n(k)-\Usorted_n(k-1)&=\xi_{\finsorted_s(k-1)+1}+\cdots+\xi_{\finsorted_s(k)} \, ,\\
&\,\,\,\vdots\\
\Usorted_n(n-1)-\Usorted_n(n-2)&=\xi_{\finsorted_s(n-2)+1}+\cdots+\xi_{\finsorted_s(n-1)} \, .
\end{align*}
In particular, conditioned on the event $S=s$, the variables $\Usorted_n(k)-\Usorted_n(k-1)$, $k=1,\dots, n-1$, are independent and have density
\newcommand{\convolutionop}{\mathop{\mathlarger{\mathlarger{\mathlarger{*}}}}}
\[
p_{\Usorted_n(k)-\Usorted_n(k-1)|S=s} (x)=\left(\convolutionop_{j=\finsorted_s(k-1)+1}^{\finsorted_s(k)} 
E_{\deg_s(j)}\right) \left(x\right), 
\]
where the notation $\convolutionop\limits_{j=1}^m f_j$ is a shorthand for the convolution $f_1 * \ldots * f_m$ of one-dimensional densities and $E_\rho(x)= \rho e^{-\rho x} \1_{[0,\infty)}(x)$ is the exponential density with parameter $\rho>0$. 
We conclude that the density of $\bm{U}_n$ conditioned on $S=s$ is
\begin{equation}
\label{eq:joint-density-un_conditioned}
p_{\bm{U}_n|S=s}(\bm{u})
=
\1_{W_{\pi_s}}(\bm{u})
\prod_{k=1}^{n-1}  \left(\convolutionop_{j=\finsorted_s(k-1)+1}^{\finsorted_s(k)} 
E_{\deg_s(j)}\right) \left(u_{\pi_s^{-1}(k)}-u_{\pi_s^{-1}(k-1)}\right) ,
\end{equation}
with the convention that $u_0:=0$ and, for any $\gamma\in S_{n-1}$, $\gamma(0):=0$.

An analogous construction holds for the continuous-time random walk on $\young(\staircase_n)$.
\emph{Mutatis mutandis}, we thus obtain that
\begin{equation}
\label{eq:joint-density-vn_conditioned}
p_{\bm{V}_n|T=t}(\bm{v})
=
\1_{W_{\sigma_t}}(\bm{v})
\prod_{k=1}^{n-1}\left(\convolutionop_{j=\diagsorted_t(k-1)+1}^{\diagsorted_t(k)} 
E_{\deg_t(j)} \right) \left(v_{\sigma_t^{-1}(k)}-v_{\sigma_t^{-1}(k-1)}\right) ,
\end{equation}
with the convention that $v_0:=0$.

\subsubsection*{Probability densities of $\bm{U}_n$ and $\bm{V}_n$}

Putting together~\eqref{eq:deg_s} with~\eqref{eq:joint-density-un_conditioned} and~\eqref{eq:deg_t} with~\eqref{eq:joint-density-vn_conditioned}, the formulas for  the density functions of $\bm{U}_n$ and of $\bm{V}_n$  take the form
\begin{align*}
p_{\bm{U}_n}(\bm{u})
&= 
\sum_{s \in \sort_n} \frac{\1_{W_{\pi_s}}(\bm{u})}{\displaystyle\prod_{j=0}^{N-1} \deg_s(j)}
\prod_{k=1}^{n-1}  \left(\convolutionop_{j=\finsorted_s(k-1)+1}^{\finsorted_s(k)} 
E_{\deg_s(j)}\right) \left(u_{\pi_s^{-1}(k)}-u_{\pi_s^{-1}(k-1)}\right) ,\nonumber
\\
p_{\bm{V}_n}(\bm{v})
&= 
\!\!\!\sum_{t \in \syt(\staircase_n)} 
\frac{\1_{W_{\sigma_t}}(\bm{v})}{\displaystyle\prod_{j=0}^{N-1} \deg_t(j) }\displaystyle
\prod_{k=1}^{n-1}\left(\convolutionop_{j=\diagsorted_t(k-1)+1}^{\diagsorted_t(k)} 
E_{\deg_t(j)} \right) \left(v_{\sigma_t^{-1}(k)}-v_{\sigma_t^{-1}(k-1)}\right) .
\end{align*}
Notice that the indicator functions of the Weyl chambers may be dropped, due to the support $[0,\infty)$ of the exponential densities; however, we keep them in the formulas for later convenience.

\begin{ex}
For $n=4$, using the parameters $\fin_s$, $\pi_s$ and $\deg_s$ from Fig.~\ref{fig:syt-sn4}, we can deduce explicit formulas for $p_{\bm{U}_4}(u_1,u_2,u_3)$ in every Weyl chamber.
Using the same colors as in Fig.~\ref{fig:syt-sn4}, we have, e.g., that
\begin{align*}
\begin{split}
p_{\bm{U}_4}({\color{colorone}u_1},{\color{colortwo}u_2},{\color{colorthree}u_3})
&=\left[E_{\color{colorone}3}*E_{\color{colorone}2}*E_{\color{colorone}2}({\color{colorone}u_1})\right]\left[E_{\color{colortwo}2}*E_{\color{colortwo}1}({\color{colortwo}u_2}-{\color{colorone}u_1})\right]\left[E_{\color{colorthree}1}({\color{colorthree}u_3}-{\color{colortwo}u_2})\right]\\
&\quad +2\left[E_{\color{colorone}3}*E_{\color{colorone}2}*E_{\color{colorone}2}*E_{\color{colorone}1}({\color{colorone}u_1})\right]\left[E_{\color{colortwo}1}({\color{colortwo}u_2}-{\color{colorone}u_1})\right]\left[E_{\color{colorthree}1}({\color{colorthree}u_3}-{\color{colortwo}u_2})\right]
\end{split}
\intertext{if ${\color{colorone}u_1}\leq {\color{colortwo}u_2}\leq {\color{colorthree}u_3}$, whereas}
p_{\bm{U}_4}({\color{colorone}u_1},{\color{colortwo}u_2},{\color{colorthree}u_3})
&=2\left[E_{\color{colorone}3}*E_{\color{colorone}2}*E_{\color{colorone}2}*E_{\color{colorone}1}({\color{colortwo}u_2})\right]\left[E_{\color{colortwo}2}({\color{colorone}u_1}-{\color{colortwo}u_2})\right]\left[E_{\color{colorthree}1}({\color{colorthree}u_3}-{\color{colorone}u_1})\right]
\end{align*}
if ${\color{colortwo}u_2}\leq {\color{colorone}u_1}\leq {\color{colorthree}u_3}$.
Considering all these $3!$ expressions, and evaluating the convolutions of exponential densities, one obtains that
\[
\begin{split}
p_{\bm{U}_4} & (u_1,u_2,u_3) \\
&=\begin{cases}
\e^{-\left(u_1+u_2+u_3\right)}\left[\e^{u_1+u_2}-(u_1-1)\e^{u_1}-(u_1+1)\e^{u_2}-1\right]&\text{if $u_1\leq u_2\leq u_3$,}\\
\e^{-\left(u_1+u_2+u_3\right)}\left[\e^{u_2}-2u_2\e^{u_2}-1\right]&\text{if $u_2\leq u_1\leq u_3$,}\\
\e^{-\left(u_1+u_2+u_3\right)}\left[\e^{u_1+u_3}-(u_1-1)\e^{u_1}-(u_1+1)\e^{u_3}-1\right]&\text{if $u_1\leq u_3\leq u_2$,}\\
\e^{-\left(u_1+u_2+u_3\right)}\left[\e^{u_2}-2u_2\e^{u_2}-1\right]&\text{if $u_2\leq u_3\leq u_1$,}\\
\e^{-\left(u_1+u_2+u_3\right)}\left[\e^{u_1+u_3}-(u_3-1)\e^{u_3}-(u_3+1)\e^{u_1}-1\right]&\text{if $u_3\leq u_1\leq u_2$,}\\
\e^{-\left(u_1+u_2+u_3\right)}\left[\e^{u_2+u_3}-(u_3-1)\e^{u_3}-(u_3+1)\e^{u_2}-1\right]&\text{if $u_3\leq u_2\leq u_1$.}
\end{cases}
\end{split}
\]
Similarly, one can compute $p_{\bm{V}_4}$, using the data $\diag_t$, $\sigma_t$ and $\deg_t$ (or, alternatively, using the recursion~\eqref{eq:jointDensityRecursive}) and check that $p_{\bm{U}_4}=p_{\bm{V}_4}$.
\end{ex}

\subsubsection*{Fourier transforms and Weyl chambers}

The conjectural equality $p_{\bm{U}_n} = p_{\bm{V}_n}$ of the joint density functions of $\bm{U}_n$ and $\bm{V}_n$ is equivalent to the equality $\hat{p_{\bm{U}_n}} = \hat{p_{\bm{V}_n}}$ of their corresponding Fourier transforms.
In turn, the latter can be manipulated and recast as the combinatorial identity~\eqref{eq:comb-iden} of Conjecture~\ref{main-conj-reformulated}.
We now outline the calculations.

Recalling the notation $W_{\gamma}$ for the Weyl chamber associated to a permutation $\gamma\in S_{n-1}$, as in~\eqref{eq:WeylChamber}, we observe that the identity $p_{\bm{U}_n} = p_{\bm{V}_n}$ is equivalent to the $(n-1)!$ equalities
\begin{equation}
\label{eq:p_q}
p_{\bm{U}_n}(\bm{z}) \1_{W_{\gamma}}(\bm{z})
=p_{\bm{V}_n}(\bm{z})\1_{W_{\gamma}}(\bm{z}) \, , \qquad\quad \gamma\in S_{n-1} \, .
\end{equation}
Introduce the change of variables
\begin{equation}
\label{eq:changeVar}
\Gamma_{\gamma}\colon \R^{n-1}_{\geq 0}\to W_{\gamma}\, ,\qquad  
\bm{z}
\mapsto
\bm{\zeta}=\Gamma_{\gamma}(\bm{z})
\end{equation}
defined by  setting 
\[
\zeta_k = z_1 + \dots + z_{\gamma(k)}\qquad \text{for } 1\leq k\leq n-1 \, .
\]
Notice that for all permutations $\gamma\in S_n$, $\Gamma_{\gamma}$ is a bijection with inverse 
\begin{equation}
\label{eq:changeVar2}
\Gamma_{\gamma}^{-1}\colon W_{\gamma}\to  \R^{n-1}_{\geq 0}\, ,\qquad  
\bm{\zeta}\mapsto\bm{z}=\Gamma_{\gamma}^{-1}(\bm{\zeta})
\end{equation}
 given by
\[
z_1= \zeta_{\gamma^{-1}(1)}
\qquad\text{and}\qquad
z_k = \zeta_{\gamma^{-1}(k)} - \zeta_{\gamma^{-1}(k-1)} \qquad \text{for } 2\leq k\leq n-1 \, .
\]
Therefore, \eqref{eq:p_q} are equivalent to the $(n-1)!$ equalities
\begin{equation}
\label{eq:p_q_gamma}
q_{\bm{U}_n}^{\gamma}(\bm{z})
=q_{\bm{V}_n}^{\gamma}(\bm{z}) \, , \qquad\quad \gamma\in S_{n-1} \, ,
\end{equation}
where 
\begin{align*}
q_{\bm{U}_n}^{\gamma}(\bm{z})
&:= p_{\bm{U}_n}(\Gamma_{\gamma}(\bm{z}))\1_{\R^{n-1}_{\geq 0}}(\bm{z}) \, , \\
q_{\bm{V}_n}^{\gamma}(\bm{z})
&:= p_{\bm{V}_n}(\Gamma_{\gamma}(\bm{z}))\1_{\R^{n-1}_{\geq 0}}(\bm{z})  \, .
\end{align*}

Now, the identities~\eqref{eq:p_q_gamma} are equivalent to the equalities of the corresponding Fourier transforms.
Using the explicit expression for the density of $\bm{U}_n$, the Fourier transform of $q_{\bm{U}_n}^{\gamma}$ can be written as
\[
\begin{split}
\hat{q_{\bm{U}_n}^{\gamma}} (x_1,\ldots, x_{n-1})
&=\int_{\R^{n-1}} q_{\bm{U}_n}^{\gamma}(z_1,\ldots,z_{n-1}) \prod_{k=1}^{n-1}e^{-\i x_k z_k}\diff z_k \\
&=\int_{\R^{n-1}} p_{\bm{U}_n}(\Gamma_{\gamma}(\bm{z}))\1_{\R^{n-1}_{\geq 0}}(\bm{z}) \prod_{k=1}^{n-1}e^{-\i x_k z_k}\diff z_k \\
&=\sum_{s \in \sort_n}
\int_{\R^{n-1}}
\prod_{k=1}^{n-1}  \left(\convolutionop_{j=\finsorted_s(k-1)+1}^{\finsorted_s(k)} 
E_{\deg_s(j)}\right) \left(\Gamma^{-1}_{\pi_s}(\Gamma_{\gamma}(\bm{z}))\right) \\
&\qquad\qquad\qquad\, \times
\frac{\1_{W_{\pi_s}}(\Gamma_{\gamma}(\bm{z}))}{\displaystyle\prod_{j=0}^{N-1} \deg_s(j)} \1_{\R^{n-1}_{\geq 0}}(\bm{z}) \prod_{k=1}^{n-1}e^{-\i x_k z_k}\diff z_k \, .
\end{split}
\]
Observe now that, when $\bm{z} \in \R_{\geq 0}^{n-1}$,
\[
\Gamma_{\gamma}(\bm{z})\in W_{\pi_s}
\quad \Longleftrightarrow \quad
\pi_s = \gamma \, .
\]
Applying the convolution theorem and the fact that the Fourier transform of the exponential density is 
\[
\hat{E_\rho}(x)
:=\int_{\R} E_\rho(u)e^{-\i xu}\diff u
=\frac{\rho}{\rho+\i x} \, ,
\]
we then continue the above computation:
\[
\begin{split}
\hat{q_{\bm{U}_n}^{\gamma}} (x_1,\ldots, x_{n-1})
&=\sum_{s \in \sort_n} \frac{\1_{\{\pi_s=\gamma\}}}{\displaystyle\prod_{j=0}^{N-1} \deg_s(j)} \prod_{k=1}^{n-1}  \int_{\R} \! \left(\convolutionop_{j=\finsorted_s(k-1)+1}^{\finsorted_s(k)} 
E_{\deg_s(j)}\right)\!\left(z_k\right)  \e^{-\i x_k z_k}\diff z_k\\
&= \sum_{s \in \sort_n} \frac{\1_{\{\pi_s=\gamma\}}}{\displaystyle\prod_{j=0}^{N-1} \deg_s(j)} \prod_{k=1}^{n-1}  \prod_{j=\finsorted_s(k-1)+1}^{\finsorted_s(k)} \int_{\R}
E_{\deg_s(j)}\left(z_k\right) \e^{-\i x_kz_k}\diff z_k\\
&= \sum_{s \in \sort_n} \1_{\{\pi_s=\gamma\}} \prod_{k=1}^{n-1}  \prod_{j=\finsorted_s(k-1)+1}^{\finsorted_s(k)} \frac{1}{\deg_s(j)+\i x_k} \, .
\end{split}
\]
Similarly, the expression for the density of $\bm{V}_n$ yields
\[
\begin{split}
\hat{q_{\bm{V}_n}^{\gamma}}(x_1,\ldots,x_{n-1})
&=\int q_{\bm{V}_n}^{\gamma}(z_1,\ldots,z_{n-1})\prod_{k=1}^{n-1}e^{-\i x_k z_k}\diff z_k\\
&=\sum_{t \in \syt(\staircase_n)} \1_{\{\sigma_t=\gamma\}}\;\prod_{k=1}^{n-1}  \prod_{j=\diagsorted_t(k-1)+1}^{\diagsorted_t(k)} \frac{1}{\deg_t(j)+\i x_k} \, .
\end{split}
\]
Replacing each $x_k$ with $-\i x_k$ in the expressions for $\hat{q_{\bm{U}_n}^{\gamma}}$ and $\hat{q_{\bm{V}_n}^{\gamma}}$, we recognize the generating factors $g_s$ and $f_t$ from~\eqref{eq:g_s} and~\eqref{eq:f_t}, respectively.
We thus conclude that the equality $p_{\bm{U}_n}=p_{\bm{V}_n}$ is equivalent to the $(n-1)!$ identities
\[
\sum_{s \in \sort_n} \1_{\{\pi_s=\gamma\}}\;g_s(x_1,\ldots,x_{n-1})=
\sum_{t \in \syt(\staircase_n)} \1_{\{\sigma_t=\gamma\}}\;f_t(x_1,\ldots,x_{n-1})\, ,\qquad \gamma\in S_{n-1} \, .
\]
These can be written more compactly as the equality of the generating functions $F_n$ and $G_n$ defined in \eqref{eq:genfun-F}-\eqref{eq:genfun-G}, that is, the relation \eqref{eq:comb-iden}.

\vskip 4mm

\noindent {\bf Acknowledgments.}
Elia Bisi was supported by ERC Grant \mbox{\emph{IntRanSt}} - 669306.
Fabio Deelan Cunden was supported by ERC Grant \mbox{\emph{IntRanSt}} - 669306 and GNFM-INdAM.
Shane Gibbons was supported by the 2019 Undergraduate Summer Research programme of the School of Mathematics and Statistics, University College Dublin.
Dan Romik was supported by the National Science Foundation grant No.\ DMS-1800725.

\appendix

\section{The RSK and Burge correspondences}
\label{app:RSK_Burge}

In this appendix we translate the results of~\cite{krattenthaler06} into Theorem~\ref{thm:RSK}.

We identify a Young diagram $\lambda$ with the sequence $(m_i,n_i)_{i=1}^{k-1}$ of its border boxes, ordered so that $m_i \geq m_{i-1}$ and $n_i \leq n_{i-1}$ for all $2\leq i\leq k-1$.
Such a sequence forms a directed `line-to-line' path, i.e.\ a directed path starting on the line $\{(i,j)\in\N^2\colon i=1\}$ and ending on the line $\{(i,j)\in\N^2\colon j=1\}$.
In other words, we have that $m_1=1$, $n_{k-1}=1$, and each increment $w_i := (m_i,n_i) -(m_{i-1}, n_{i-1})$ is either $D:=(0,-1)$ or $R:=(1,0)$ for all $2\leq i\leq k-1$\footnote{Letters $D$ and $R$ refer to \emph{down} and \emph{right} steps of the directed path, respectively, if one uses the French notation for Young diagrams as in~\cite{krattenthaler06}.
In the English translation, which we have used throughout this paper, $D$ and $R$ correspond to left and down steps, respectively.}.
Setting also by convention $w_1:=R$ and $w_k:= D$, one can identify $\lambda$ with a `$D$-$R$-sequence' $w=w_1\dots w_k$ starting at $R$ and ending at $D$.
For instance, the shape of the tableaux in Fig.~\ref{fig:RSK-Burge} is encoded as the sequence $RDRRRDDRD$.

Given a partition $\lambda$ associated with a $D$-$R$ sequence $w=w_1\dots w_k$, \cite[Theorem~7]{krattenthaler06} describes the RSK map as a bijection between Young tableaux $x$ of shape $\lambda$ with non-negative integers entries and sequences $(\emptyset = \mu^0, \mu^1,\dots, \mu^k = \emptyset)$ of partitions such that $\mu^i / \mu^{i-1}$ is a horizontal strip if $w_i=R$ and $\mu^{i-1} / \mu^i$ is a horizontal strip if $w_i=D$.
One can easily verify that, for $1\leq i\leq k-1$, the partition $\mu^i$ is of length $p_i := \min(m_i,n_i)$ at most.
We can then form a new Young tableau $r=\{r_{i,j} \colon (i,j)\in \lambda\}$ by setting the diagonal of $r$ that contains the border box $(m_i,n_i)$ to be
\[
(r_{m_i,n_i}, r_{m_i-1,n_i-1}, \dots, r_{m_i-p_i+1, n_i - p_i +1}) := \mu^i \qquad\quad
\text{for $1\leq i\leq k-1$.}
\]
It is then easy to check that the conditions on $\mu^i / \mu^{i-1}$ and $\mu^{i-1} / \mu^i$ are equivalent to the fact that $r$ is an interlacing tableau in the sense of~\eqref{eq:interlacing}.
Therefore, the sequence $(\emptyset = \mu^0, \mu^1,\dots, \mu^k = \emptyset)$ can be rearranged into an interlacing tableau of shape $\lambda$ with non-negative integer entries, thus yielding the RSK correspondence of Theorem~\ref{thm:RSK}.
The fact that~\eqref{eq:RSK_Greene} holds follows then from~\cite[Theorem~8-(G\textsuperscript{1}1)]{krattenthaler06}.

The statement about the Burge correspondence in Theorem~\ref{thm:RSK}, which is called \emph{dual RSK'} (to be read: dual RSK prime) in~\cite{krattenthaler06}, can be recovered in a similar way from the results of that paper.
Given a partition $\lambda$ associated with a $D$-$R$ sequence $w=w_1\dots w_k$, \cite[Theorem~11]{krattenthaler06} presents the Burge correspondence as a bijection between Young tableaux $x$ of shape $\lambda$ with non-negative integer entries and sequences $(\emptyset = \nu^0, \nu^1,\dots, \nu^k = \emptyset)$ of partitions such that $\nu^i / \nu^{i-1}$ is a \emph{vertical} strip if $w_i=R$ and $\nu^{i-1} / \nu^i$ is a \emph{vertical} strip if $w_i=D$.
This time, we define the Young tableau $b=\{b_{i,j} \colon (i,j)\in \lambda\}$ by identifying the diagonal of $b$ that contains $(m_i,n_i)$ with the \emph{conjugate} partition of $\nu^i$:
\[
(b_{m_i,n_i}, b_{m_i-1,n_i-1}, \dots, b_{m_i-p_i+1, n_i - p_i +1}) := (\nu^i)' \qquad\quad
\text{for $1\leq i\leq k-1$.}
\]
This resulting map $x\mapsto b$ satisfies~\eqref{eq:Burge_Greene} thanks to~\cite[Theorem~12-(G\textsuperscript{4}2)]{krattenthaler06}.

When $\lambda$ is a rectangular shape $[1,m] \times [1,n]$, the RSK and Burge correspondences degenerate to the classical ones in the following way.
The sequence of partitions $(\emptyset = \mu^0, \mu^1,\dots, \mu^{m+n} = \emptyset)$ corresponding to an $m\times n$ matrix $x$ via RSK can be split into an ascending and a descending sequence:
\[
\emptyset = \mu^0 \subseteq \mu^1 \subseteq \dots \subseteq \mu^m \supseteq \dots \supseteq \mu^{m+n-1} \supseteq \mu^{m+n} = \emptyset \, .
\]
One can then form two Young tableaux $P$ and $Q$ of common shape $\mu^m$ by setting $Q_{i,j} := k$ if and only if $(i,j)\in \mu^k / \mu^{k-1}$ for all $1\leq k\leq m$ and $P_{i,j} := l$ if and only if $(i,j)\in \mu^{m+n-l} / \mu^{m+n-l+1}$ for all $1\leq l\leq n$.
The constraint on the partitions make the two tableaux $P$ and $Q$ semistandard, and the map $x\mapsto (P,Q)$ corresponds to the classical RSK correspondence.
An analogous connection with the classical Burge correspondence also holds.

\vskip 4mm

\printbibliography

\end{document}